\documentclass[a4paper,10pt]{article}

\usepackage{titlesec}

\titleformat{\section}
{\normalfont\large\bfseries}{\thesection}{1em}{}
\titleformat{\subsection}
{\normalfont\normalsize\bfseries}{\thesubsection}{1em}{}
\titleformat{\subsubsection}
{\normalfont\normalsize\bfseries}{\thesubsubsection}{1em}{}
\titleformat{\paragraph}[runin]
{\normalfont\normalsize\bfseries}{\theparagraph}{1em}{}
\titleformat{\subparagraph}[runin]
{\normalfont\normalsize\bfseries}{\thesubparagraph}{1em}{}

\usepackage{amssymb,amsmath,amsthm}
\usepackage{mathrsfs}
\usepackage{amsmath}
\PassOptionsToPackage{hyphens}{url}\usepackage[hidelinks]{hyperref}
\usepackage{cleveref}
\usepackage{mathtools}
\usepackage{xspace}
\usepackage{ifthen}
\usepackage{enumitem}
\usepackage{todonotes}
\usepackage{graphicx}
\usepackage{subfig}

\makeatletter
\DeclareOldFontCommand{\sc}{\normalfont\scshape}{\@nomath\sc}
\makeatother

\crefname{section}{Sect.}{Sects.}
\Crefname{section}{Section}{Sections}

\newtheorem{Theorem}{Theorem}
\newtheorem{proposition}{Proposition}
\newtheorem{lemma}{Lemma}
\newtheorem{assumption}{Assumption}
\newtheorem{example}{Example}
\crefname{enumi}{Assumption}{Assumptions}

\Crefrangeformat{enumi}{Assumptions #3#1#4--#5#2#6}
\crefname{corollary}{Corollary}{Corollaries}
\crefname{proposition}{Proposition}{Propositions}
\Crefname{assumption}{Assumption}{Assumptions}
\Crefrangeformat{assumption}{Assumptions #3#1#4--#5#2#6}
\newlist{enumthm}{enumerate}{1}
\setlist[enumthm]{label=\textup{(\alph*)},ref=\theassumption~\textup{(\alph*)}}
\crefalias{enumthmi}{assumption}

\newcommand{\hsp}{H}
\newcommand{\csp}{U}
\newcommand{\ssp}{Y}
\newcommand{\asp}{Z}
\newcommand{\adcsp}{U_{\text{ad}}}
\newcommand{\bsp}{V}
\newcommand{\nbsp}{V}
\newcommand{\nnbsp}{W}
\newcommand{\rv}{Z}
\newcommand{\rvariable}{X}

\newcommand{\randomsum}{S}
\newcommand{\norm}[2][2]{\|#2\|_{#1}}

\DeclarePairedDelimiterXPP\cnorm[2]{}\lVert\rVert{_{#1}}{#2}
\renewcommand{\natural}{\mathbb{N}}
\newcommand{\tfa}{\text{for all}}
\newcommand{\tand}{\text{and}}
\newcommand{\real}{\mathbb{R}}

\newcommand{\eu}{\ensuremath{\mathrm{e}}}
\newcommand{\wpone}{with probability one}
\newcommand{\inner}[3][]{( #2, #3 )_{#1}}
\DeclarePairedDelimiterX\innerp[3]{(}{)_{#1}}{#2,#3}

\newcommand{\pOmega}{\Omega}
\newcommand{\cF}{\mathcal{F}}

\newcommand{\cB}{\mathcal{B}}
\newcommand{\rdc}{\kappa}

\newcommand{\domain}{D}
\newcommand{\rrhs}{h}

\newcommand{\du}{\ensuremath{\mathrm{d}}}

\newcommand{\sblf}[2][\real]{\mathscr{L}(#2, #1)}
\newcommand{\embedding}{\xhookrightarrow{}}
\newcommand{\wrt}{with respect to}
\newcommand{\dom}[1]{\mathrm{dom}~{#1}}
\newcommand{\dualpHzeroone}[3][]
{\langle #2, #3 \rangle_{H^{-1}(#1), H_0^1(#1)}}

\newcommand{\dualp}[3][]{\langle #2, #3 \rangle_{{#1}^*\!, #1}}

\DeclarePairedDelimiterXPP\cE[1]{\mathbb{E}}[]{}{
	
	#1}
\DeclarePairedDelimiterXPP\Prob[1]{\mathrm{Prob}}(){}{
	
	#1}

\newcommand{\MonteArg}[1]{\ifthenelse{\equal{#1}{}}  
	{}
	{[{#1}]}
}

\newcommand{\erpobj}{F}
\newcommand{\obj}{f}
\newcommand{\pobj}{J}
\newcommand{\rpobj}{\pobj}

\newcommand{\fun}{\mathsf{f}}

\newcommand{\gateaux}{G\^ateaux}
\newcommand{\Caratheodory}{Carath\'eodory}

\newcommand{\youngfun}[1]{\phi}
\newcommand{\luxemburg}[3][B]{\norm[L_{\youngfun{#2}}(\Omega; #1)]{#3}}

\newcommand{\dn}{64}
\newcommand{\nn}{256}

\title{Sample Average Approximations of Strongly Convex Stochastic Programs in Hilbert Spaces}
\date{May 21, 2022}
\author{Johannes Milz\thanks{Technical University of Munich,
		Department of Mathematics, Boltzmannstr.\ 3, 85748 Garching, Germany,
		\texttt{milz@ma.tum.de}. The
		project was funded by the Deutsche Forschungsgemeinschaft 
		(DFG, German Research Foundation)--project number 188264188/GRK1754.}}

\begin{document}
\maketitle
\begin{abstract}
  We analyze
  the tail behavior of  solutions to sample average approximations (SAAs)
  of stochastic programs posed in Hilbert spaces. 
  We require that the integrand be strongly convex
  with the same convexity parameter for each realization.
  Combined with a standard condition
  from the literature on stochastic programming, 
  we  establish non-asymptotic exponential tail bounds 
  for the distance between the SAA solutions and 
  the stochastic program's solution,
  without assuming compactness of the feasible set. 
  Our assumptions are verified on a class of 
  infinite-dimensional optimization problems 
  governed
  by affine-linear partial differential equations with random inputs.
  We present numerical results illustrating our theoretical findings.
\end{abstract}

\section{Introduction}
\label{sec:intro}

We apply the sample average approximation (SAA)
to a class of strongly convex stochastic programs posed
in Hilbert spaces, and study the tail behavior of the distance
between SAA solutions and their true counterparts. 
Our work sheds light on the number of samples needed
to reliably estimate solutions to infinite-dimensional, 
linear-quadratic optimal control problems
governed by affine-linear partial differential equations (PDEs)
with random inputs, a class of 
optimization problems that has received much attention
recently
\cite{Hoffhues2020,Marin2018,MartnezFrutos2018}.
Our analysis requires that the integrand
be strongly convex with the same convexity parameter
for each random  element's sample. 
This assumption is fulfilled
for convex optimal controls problems with a strongly convex
control regularizer, such as those considered in
\cite{Hoffhues2020,Marin2018,MartnezFrutos2018}.
Throughout the paper, 
a function $\mathsf{f} : \hsp \to \real \cup \{\infty\}$
is $\alpha$-strongly convex with parameter $\alpha  > 0$
if $\mathsf{f}(\cdot)-(\alpha/2)\norm[\hsp]{\cdot}^2$
is convex,
where $\hsp$ is a real Hilbert space with norm $\norm[\hsp]{\cdot}$.
Moreover, a function on a real Hilbert space is strongly convex if it
is $\alpha$-strongly convex with some parameter $\alpha > 0$.

We consider the potentially infinite-dimensional
stochastic program
\begin{align}
\label{eq:saa:2020-01-15T21:11:01.962}
\min_{u \in \csp}\, 
\{\, \obj(u) = \cE{\rpobj(u, \xi)} + \Psi(u) \, \}, 
\end{align}
where $\csp$ is a real, separable Hilbert space, 
$\Psi : \csp \to \real \cup \{\infty\}$
is proper, lower-semicontinuous and convex,
and $\rpobj : \csp \times \Xi \to \real$ is
the integrand.
Moreover, $\xi$ is a random element
mapping from a probability space to 
a complete, separable metric space $\Xi$
equipped with its Borel $\sigma$-field.
We also use $\xi \in \Xi$ to represent a deterministic element.

Let $\xi^1$, $\xi^2$, \ldots be  independent  identically 
distributed $\Xi$-valued random elements
defined on a complete probability space $(\Omega, \cF, P)$
such that each $\xi^i$
has the same distribution as that of $\xi$.
The SAA problem corresponding to \eqref{eq:saa:2020-01-15T21:11:01.962} is
\begin{align}
\label{eq:saa:saalqcp}
\min_{u \in \csp }\, 
\Big\{\, 
\obj_N(u) =  
\frac{1}{N}\sum_{i=1}^N \rpobj(u, \xi^i) + \Psi(u)
\, \Big\},
\end{align}
We define $\erpobj: \csp \to \real \cup \{\infty\}$ 
and the sample average function
$\erpobj_N : \csp \to \real$
by 
\begin{align}
\label{eq:saa:erpobjN}
\erpobj(u) = \cE{\rpobj(u, \xi)}
\quad \tand \quad 
\erpobj_N(u) =
\frac{1}{N}\sum_{i=1}^N \rpobj(u, \xi^i).
\end{align}
Since we assume
that the random elements
$\xi^1, \xi^2, \ldots$
are defined on the common 
probability space $(\Omega, \cF, P)$,
we can view the functions $\obj_N$ and $\erpobj_N$
as defined on $\csp \times \Omega$
and the solution $u_N^*$ to \eqref{eq:saa:saalqcp} 
as a mapping from $\Omega$ to $\csp$.
The second argument of $\obj_N$ and of $\erpobj_N$ is often dropped.

Let $u^*$ be a solution to \eqref{eq:saa:2020-01-15T21:11:01.962}
and $u_N^*$ be a solution to \eqref{eq:saa:saalqcp}.
We assume that $\rpobj(\cdot, \xi)$ is $\alpha$-strongly convex
with parameter $\alpha > 0$ for each $\xi \in \Xi$.
Furthermore, we assume that
$\erpobj(\cdot)$ and
$\rpobj(\cdot, \xi)$ 
for all $\xi \in \Xi$
are \gateaux\ differentiable.
Under these assumptions, we establish
the error estimate
\begin{align}
\label{eq:intro:2020-01-25T19:33:12.799}
\alpha \norm[\csp]{u^*-u_N^*}
\leq \cnorm{\csp}{\nabla \erpobj_N(u^*)- \nabla \erpobj(u^*)},
\end{align}
valid \wpone.
If $\norm[\csp]{\nabla_u\rpobj(u^*, \xi)}$
is integrable, then 
$\nabla \erpobj_N(u^*)$ is just
the empirical mean of $\nabla \erpobj(u^*)$
since $\erpobj(\cdot)$ and $\rpobj(\cdot, \xi)$
for all $\xi \in \Xi$ are convex and
\gateaux\ differentiable at $u^*$; see  \Cref{lem:saa:2020-04-01T18:25:27.402}.
Hence we can analyze the mean square error
$\cE{\norm[\csp]{u^*-u_N^*}^2}$ and the 
exponential tail behavior of $\norm[\csp]{u^*-u_N^*}$
using standard conditions from the
literature on stochastic programming.
To obtain a bound on $\cE{\norm[\csp]{u^*-u_N^*}^2}$, we assume that
there exists $\sigma > 0$ with
\begin{align}
\label{eq:sigma}
\cE{\norm[\csp]{\nabla_u \rpobj(u^*, \xi)-
		\nabla \erpobj(u^*)}^2} \leq \sigma^2,
\end{align}
yielding with \eqref{eq:intro:2020-01-25T19:33:12.799} the bound 
\begin{align}
\label{eq:meansquarederror}
\cE{\norm[\csp]{u^*-u_N^*}^2} \leq \sigma^2/(\alpha^2 N).
\end{align}
To derive exponential tail bounds on $\norm[\csp]{u^*-u_N^*}$, 
we further assume the existence of  $\tau > 0$ with
\begin{align}
\label{eq:exponentialsquareintegrable}
\cE{\exp 
	(
	\tau^{-2}\norm[\csp]{\nabla_u \rpobj(u^*, \xi)-
		\nabla\erpobj(u^*)}^2
	)
}
\leq \eu.
\end{align}
This condition and its variants 
are used, for example,  in
\cite{Duchi2012,Guigues2017,Nemirovski2009}.
Using Jensen's inequality, we find that \eqref{eq:exponentialsquareintegrable} 
implies \eqref{eq:sigma} with $\sigma^2 = \tau^2$
\cite[p.\ 1584]{Nemirovski2009}.
Combining \eqref{eq:intro:2020-01-25T19:33:12.799} and
\eqref{eq:exponentialsquareintegrable}
with the exponential moment 
inequality proven in \cite[Thm.\ 3]{Pinelis1986},
we establish the exponential tail bound, our
main contribution,
\begin{align}
\label{eq:exponentialtailbound}
\Prob{\norm[\csp]{u^*-u_N^*} \geq \varepsilon}
\leq 
2\exp(-\tau^{-2}N \varepsilon^2 \alpha^2/3)
\quad \tfa \quad \varepsilon > 0.
\end{align} 
This bound solely depends on the characteristics
of $\rpobj$ but not on properties of
the feasible set,
$\{\, u \in \csp \colon \, \Psi(u) < \infty \, \}$,
other than its convexity.
For each $\delta \in (0,1)$, 
the exponential tail bound yields,
with a probability of at least $1-\delta$,
\begin{align}
\label{eq:intro:exponentialtailbound}
\norm[\csp]{u^*-u_N^*} <
\frac{\tau}{\alpha}\sqrt{\frac{3\ln(2/\delta)}{N}}.
\end{align}
In particular,
if $\varepsilon > 0$
and $N \geq \tfrac{3\tau^2}{\alpha^2\varepsilon^2}\ln(2/\delta)$, 
then $\norm[\csp]{u^*-u_N^*} < \varepsilon$
with a probability of at least $1-\delta$, 
that is, $u^*$ can be estimated reliably
via $u_N^*$.

Requiring  $\rpobj(\cdot, \xi)$ to be $\alpha$-strongly convex
for each $\xi \in \Xi$ is a restrictive assumption. However, it
is fulfilled for the following class of stochastic programs:
\begin{align}
\label{eq:intro:lqcp}
\min_{u \in \csp} \, 
\{\, 
(1/2)\cE{\norm[\hsp]{K(\xi)u+h(\xi)}^2} +
(\alpha/2)\norm[\csp]{u}^2
+ \Psi(u)
\, \},
\end{align}
where $\alpha > 0$, $\hsp$ and $\csp$ are real Hilbert spaces, and
$K(\xi) : \csp \to \hsp$ is a bounded, linear operator and $h(\xi) \in \hsp$
for each $\xi \in \Xi$.
The control problems
governed by affine-linear PDEs with random inputs
considered, for example, in
\cite{Ge2018,Geiersbach2019a,Hoffhues2020,Martin2021,%
	MartnezFrutos2018}
can be formulated as instances of \eqref{eq:intro:lqcp}.
In many of  these works, the operator $K(\xi)$ is compact
for each $\xi \in \Xi$, 
the expectation function $\erpobj_1 : \csp \to \real$
defined by
$\erpobj_1(u) = (1/2)\cE{\norm[\hsp]{K(\xi)u+h(\xi)}^2}$
is twice continuously differentiable,
and $\csp$ is infinite-dimensional.
In this case, the function $\erpobj_1$
generally lacks  strong convexity.
This may suggest that the $\alpha$-strong convexity of the
objective function of \eqref{eq:intro:lqcp} is
solely implied by the function
$(\alpha/2)\norm[\csp]{\cdot}^2 + \Psi(\cdot)$.
The lack of the expectation function's strong convexity
is essentially known \cite[p.\ 3]{Bach2011}. For example, if
the set $\Xi$ has finite cardinality, then
the Hessian $\nabla^2 \erpobj_1(0)$
is the finite sum of compact operators and 
hence $\erpobj_1$ lacks strong convexity; see
\cref{subsec:saa:lqocu}.

A common notion used to analyze 
the SAA solutions is that of an
$\varepsilon$-optimal solution
\cite{Shapiro2003,Shapiro2014,Shapiro2005}.\footnote{%
	A point $\bar x \in X$ is an $\varepsilon$-optimal solution 
	to $\inf_{x \in X}\, \fun(x)$ if
	$\fun(\bar x) \leq \inf_{x \in X}\, \fun(x)+\varepsilon$.
} 
We instead study the tail behavior of $\norm[\csp]{u^*-u_N^*}$
since in the literature on PDE-constrained optimization the focus is on
studying the proximity of approximate solutions to the 
``true'' ones.
For example, 
when analyzing finite element approximations of PDE-constrained problems,
bounds on the error $\norm[\csp]{w^*-w_h^*}$ as functions of the
discretization parameters $h$ are often 
established \cite{Hinze2005a,Troeltzsch2010},
where  $w^*$ is the solution to
a control problem and
$w_h^*$ is the solution to
its finite element approximation.
The estimate \eqref{eq:intro:2020-01-25T19:33:12.799} 
is similar
to that established in
\cite[p.\ 49]{Hinze2005a}
for the variational discretization---a finite element approximation---of a 
deterministic, linear-quadratic control problem. 
Since both the variational discretization and the
SAA approach yield perturbed optimization problems, 
it is unsurprising that
similar techniques can be used 
for some parts of the perturbation analysis.

The SAA approach has  thoroughly  been analyzed, for example, in 
\cite{Artstein1995,Banholzer2019,Royset2019,Shapiro2003,%
	Shapiro2005,Shapiro2014}. 
Some consistency results
for the SAA solutions and  finite-sample size estimates 
require the compactness and total boundedness of the
feasible set, respectively. However, 
in the literature on PDE-constrained optimization, 
the feasible sets are commonly noncompact; 
see, e.g., \cite[Sect.\ 1.7.2.3]{Hinze2009}.
Assuming that the function $\erpobj$ defined in
\eqref{eq:saa:erpobjN} is $\alpha$-strongly convex with $\alpha > 0$, 
Kouri and Shapiro~\cite[eq.\ (42)]{Kouri2018} establish
\begin{align}
\label{eq:saa:2020-10-23T09:36:43.358}
\alpha \norm[\csp]{u^*-u_N^*}
\leq \cnorm{\csp}{\nabla \erpobj_N(u_N^*)- \nabla \erpobj(u_N^*)}.
\end{align}
The setting in \cite{Kouri2018} corresponds to 
$\Psi$ being the indicator function of 
a closed, convex, 
nonempty subset of $\csp$.
In contrast to  the estimate \eqref{eq:intro:2020-01-25T19:33:12.799}, 
the right-hand side in \eqref{eq:saa:2020-10-23T09:36:43.358} depends
on the random control $u_N^*$. This dependence implies that
the right-hand side in \eqref{eq:saa:2020-10-23T09:36:43.358}
is more  difficult to analyze than that 
in \eqref{eq:intro:2020-01-25T19:33:12.799}.
However, the convexity assumption on $\erpobj$ made
in \cite{Kouri2018} is weaker than ours
which requires the function $\rpobj(\cdot, \xi)$  be $\alpha$-strongly convex
for all $\xi \in \Xi$. 
The right-hand side \eqref{eq:saa:2020-10-23T09:36:43.358}
may be analyzed using the approaches developed in 
\cite[Sects.\ 2 and 4]{Shapiro1993}.

For finite-dimensional optimization problems, 
the number of samples, required to obtain 
$\varepsilon$-optimal solutions via the SAA approach, 
can explicitly depend on the problem's dimension
\cite{Agarwal2009},
\cite[Ex.\ 1]{Shapiro2008}, 
\cite[Prop.\ 2]{Guigues2017}.
Guigues, Juditsky, and Nemirovski~\cite{Guigues2017}
demonstrate that confidence bounds on the 
optimal value of stochastic, convex, finite-dimensional programs,
constructed via SAA optimal values,
do not explicitly depend on the problem's dimension. 
This property is shared by our exponential tail bound.

After the initial version of the manuscript was submitted, 
we became aware of the papers \cite{Tsybakov1981,Shalev-Shwartz2010}
where assumptions similar to those used to derive 
\eqref{eq:meansquarederror} and
\eqref{eq:exponentialtailbound}
are utilized to analyze the 
reliability of SAA solutions. 
For unconstrained minimization in $\real^n$
with $\Psi = 0$, tail bounds 
for $\norm[2]{u^*-u_N^*}$ are established in \cite{Tsybakov1981} 
under the assumption
that $\rpobj(\cdot, \xi)$ 
is  $\alpha$-strong convex for  all $\xi \in \Xi$ 
and some $\alpha > 0$.
Here, $\norm[2]{\cdot}$ is the Euclidean norm on $\real^n$.
Assuming further that  $\norm[2]{\nabla_u\rpobj(u^*,\xi)}$
is essentially bounded by $L > 0$, the author
establishes
\begin{align}
\label{eq:tsy}
\Prob{\norm[2]{u^*-u_N^*} \geq \varepsilon}
\leq 
2 \exp\Big(-\tfrac{N\alpha^2\varepsilon^2}{2L^2}
\big(1+\tfrac{\alpha\varepsilon}{3L}\big)^{-1}
\Big)
\end{align}
if $\varepsilon \in (0,L/\alpha]$, 
and the right-hand side in \eqref{eq:tsy} is zero otherwise
\cite[Cor.\ 2]{Tsybakov1981}.
While \eqref{eq:tsy} is similar to \eqref{eq:exponentialtailbound}
with $\tau = L$, 
its derivation exploits the essential boundedness of
$\norm[2]{\nabla_u\rpobj(u^*,\xi)}$
which is generally more restrictive than
\eqref{eq:exponentialsquareintegrable}.
The author establishes further tail bounds 
for $\norm[2]{u^*-u_N^*}$
under different sets of assumptions on $\rpobj(\cdot, \xi)$,
and provides exponential tail bounds for
$\obj(u_N^*) - \obj(u^*)$ 
assuming that $\rpobj(\cdot,\xi)$
is Lipschitz continuous with a Lipschitz constant
independent of $\xi$ (see \cite[Thm.\ 5]{Tsybakov1981}).
For the possibly infinite-dimensional program 
\eqref{eq:saa:2020-01-15T21:11:01.962},
similar assumptions are used in 
\cite[Thm.\ 2]{Shalev-Shwartz2010} to establish
a non-exponential tail bound for
$\obj(u_N^*) - \obj(u^*)$.
While tail bounds for $\obj(u_N^*) - \obj(u^*)$ 
are derived in \cite{Tsybakov1981,Shalev-Shwartz2010}, the assumptions
used to derive 
\eqref{eq:meansquarederror} and
\eqref{eq:exponentialtailbound}  do not imply bounds on
$\obj(u_N^*) - \obj(u^*)$.

Hoffhues, R{\"o}misch, and Surowiec~\cite{Hoffhues2020}
provide qualitative and quantitative stability results for the optimal
value and for the optimal solutions
of stochastic, linear-quadratic optimization problems posed in Hilbert spaces, 
similar to those in \eqref{eq:intro:lqcp}, 
\wrt\ Fortet--Mourier and Wasserstein metrics.
These stability results are valid for 
approximating probability measures other than the empirical one,
which is used to define the SAA problem \eqref{eq:saa:saalqcp}. 
However, the convergence rate $1/N$ for 
$\cE{\norm[\csp]{u^*-u_N^*}^2}$,
and  exponential tail bounds on $\norm[\csp]{u^*-u_N^*}$
are not established in \cite{Hoffhues2020}. 
For a class of constrained, linear 
elliptic control problems, R{\"o}misch and Surowiec~\cite{Roemisch2021}
demonstrate the consistency of the solutions
and the  optimal value, the convergence rate $1/\sqrt{N}$  for 
$\cE{\norm[\csp]{u^*-u_N^*}}$ 
and for $\cE{|\obj_N(u_N^*)-\obj(u^*)|}$, 
and the convergence in distribution of 
$\sqrt{N}(\obj_N(u_N^*)-\obj(u^*))$ to a real-valued random variable.
These results are established using empirical process theory 
and are built on smoothness of the random elliptic operator and 
right-hand side with respect to the parameters.
While our assumptions yield the mean square error bound
\eqref{eq:meansquarederror} and the exponential tail bound 
\eqref{eq:exponentialtailbound}, 
further conditions may be required to establish bounds on
$\cE{|\obj_N(u_N^*)-\obj(u^*)|}$.
A bound on $\cE{\norm[\csp]{u^*-u_N^*}^2}$ 
related to \eqref{eq:meansquarederror}
is established in \cite[Thm.\ 4.1]{Martin2021}
for class of linear elliptic control problems.

Besides considering risk-neutral, convex control problems
with PDEs which can be expressed as those in \cref{subsec:saa:lqocu}, 
the authors of \cite{Marin2018,MartnezFrutos2018}
study the minimization of
$u \mapsto \Prob{\rpobj(u,\xi) \geq \rho}$,
where $\rho \in \real$
and evaluating $\rpobj(u,\xi)$ requires solving a PDE\@.
Furthermore, the authors of \cite{Marin2018,MartnezFrutos2018}
prove the existence of
solutions and use stochastic collocation to discretize
the expected values. 
In \cite[Sect.\ 5.3]{MartnezFrutos2018}, the authors adaptively  combine 
a Monte Carlo 
sampling approach with a stochastic Galerkin finite element method 
to reduce the computational costs, but
error bounds are not established.
Stochastic collocation is also used, for example, in \cite{Ge2018,Kouri2013}.
Further approaches to discretize the expected value in 
\eqref{eq:intro:lqcp} are, for example,
quasi-Monte Carlo sampling \cite{Guth2019} and 
low-rank tensor approximations
\cite{Garreis2017}. A solution method for
\eqref{eq:saa:2020-01-15T21:11:01.962}
is (robust) stochastic approximation. It has
thoroughly been analyzed 
in \cite{Lan2020,Nemirovski2009} for  finite-dimensional
and in \cite{Geiersbach2019a,Geiersbach2019,Nemirovsky1983}
for infinite-dimensional optimization problems.
For reliable $\varepsilon$-optimal solutions, 
the sample size estimates established in 
\cite[Prop.\ 2.2]{Nemirovski2009} do not explicitly
depend on the problem's dimension.

After providing some notation and preliminaries
in \cref{sec:notation}, we establish exponential tail bounds
for Hilbert space-valued random sums in 
\cref{sec:tailbounds}. Combined 
with optimality conditions and the
integrand's $\alpha$-strong convexity,
we establish exponential tail and mean square error bounds for SAA solutions
in \cref{sec:tailboundssaa}.
\cref{sec:optimalitytailbounds} demonstrates the optimality
of the  tail bounds. 
We apply our findings to linear-quadratic control
under uncertainty in \cref{subsec:saa:lqocu}, 
and identify a problem class that violates
the  integrability condition
\eqref{eq:exponentialsquareintegrable}.
Numerical results are presented in \cref{sec:numresults}.
In  \cref{sec:discussion}, we illustrate that the
``dynamics'' of finite- and infinite-dimensional
stochastic programs can be quite different.

\section{Notation and Preliminaries}
\label{sec:notation}

Throughout the manuscript, we assume the existence
of solutions to \eqref{eq:saa:2020-01-15T21:11:01.962} and to
\eqref{eq:saa:saalqcp}. We refer the reader to
\cite[Prop.\ 3.12]{Kouri2018a}
and \cite[Thm.\ 1]{Hoffhues2020}
for theorems on the existence of solutions
to infinite-dimensional stochastic programs.

The set $\dom{\Psi} = \{\, u \in \csp \colon \,
\Psi(u) < \infty \,\}$ is the domain of $\Psi$.
The indicator function $I_{\csp_0} : \csp \to \real \cup\{\infty\}$
of a nonempty set $\csp_0 \subset \csp$ is defined by
$I_{\csp_0}(u) = 0$ if $u \in \csp_0$ and 
$I_{\csp_0}(u) =  \infty$ otherwise.
Let $(\hat{\Omega}, \hat{\cF}, \hat{P})$ be a probability space.
A Banach space $\nnbsp$ 
is equipped with its Borel $\sigma$-field
$\cB(\nnbsp)$.
We denote by $\inner[\hsp]{\cdot}{\cdot}$
the inner product of a real Hilbert space $\hsp$
equipped with the norm $\norm[\hsp]{\cdot}$
given by 
$\norm[\hsp]{v}  = \sqrt{\inner[\hsp]{v}{v}}$
for all $v \in \hsp$.
For a real, separable Hilbert space $\hsp$,
$\eta : \hat{\Omega} \to \hsp$ is a mean-zero Gaussian
random vector if $\inner[\hsp]{v}{\eta}$ is a mean-zero
Gaussian random variable for each $v \in \hsp$
\cite[pp.\ 58--59]{Yurinsky1995}.
For  a metric space $\nbsp$,
a mapping $\fun : \nbsp \times \hat{\Omega} \to \nnbsp$
is a \Caratheodory\ mapping
if $\fun(\cdot, \omega)$ is continuous for every 
$\omega \in \hat{\Omega}$ and
$\fun(x, \cdot)$ is
$\hat{\cF}$-$\cB(\nnbsp)$-measurable 
for all $x \in \nbsp$.

For two Banach spaces $\nbsp$ and $\nnbsp$,
$\sblf[\nnbsp]{\nbsp}$ is the space of bounded, linear
operators from $\nbsp$ to $\nnbsp$, 
and $\nbsp^* = \sblf{\nbsp}$.
We denote by $\dualp[\nbsp]{\cdot}{\cdot}$ the dual
pairing of $\nbsp^*$ and $\nbsp$.
A function $\upsilon: \hat{\Omega} \to \nnbsp$ is strongly measurable
if there exists a sequence of simple functions 
$\upsilon_k : \hat{\Omega}  \to \nnbsp$ such that
$\upsilon_k(\omega) \to \upsilon(\omega)$ as $k\to \infty$
for all $\omega \in \hat{\Omega}$ \cite[Def.\ 1.1.4]{Hytoenen2016}.	
An 
operator-valued function
$\Upsilon : \hat{\Omega} \to \sblf[\nnbsp]{\nbsp}$
is strongly measurable if the function
$\omega \mapsto \Upsilon(\omega)x$ is strongly measurable
for each $x \in \nbsp$ \cite[Def.\ 1.1.27]{Hytoenen2016}.
Moreover, an operator-valued function
$\Upsilon : \hat{\Omega} \to 
\sblf[\nnbsp]{\nbsp}$
is uniformly measurable if 
there exists a sequence of simple operator-valued
functions $\Upsilon_k : 
\hat{\Omega} \to \sblf[\nnbsp]{\nbsp}$
with $\Upsilon_k(\omega) \to \Upsilon(\omega)$ as $k \to \infty$
for all $\omega\in\hat{\Omega}$.
An operator 
$K \in \sblf[\nnbsp]{\nbsp}$ is compact
if the closure of $K(\nbsp_0)$ is compact for each
bounded set $\nbsp_0 \subset \nbsp$.
For two real Hilbert spaces $\hsp_1$ and $\hsp_2$,  
$K^* \in \sblf[\hsp_1]{\hsp_2}$ 
is the (Hilbert space-)adjoint operator
of $K \in \sblf[\hsp_2]{\hsp_1}$ and is defined by 
$\inner[\hsp_2]{Kv_1}{v_2} =
\inner[\hsp_1]{v_1}{K^*v_2}$
for all $v_1 \in \hsp_1$ and $v_2 \in \hsp_2$
\cite[Def.\ 3.9.1]{Kreyszig1978}.
For a bounded domain $\domain \subset \real^d$, 
$L^2(\domain)$ is the Lebesgue space of square-integrable functions
and $L^1(\domain)$ is that of integrable functions.
The Hilbert space $H_0^1(\domain)$
consists of all $v \in L^2(\domain)$
with weak derivatives in $L^2(\domain)^d$
and with zero boundary traces. 
We define $H^{-1}(\domain) = H_0^1(\domain)^*$.

\section{Exponential Tail Bounds for Hilbert Space-Valued Random Sums}
\label{sec:tailbounds}

We establish two exponential tail bounds
for Hilbert space-valued random sums which are 
direct consequences of known results
\cite{Pinelis1994,Pinelis1986}.
Below, $(\Theta, \Sigma, \mu)$
denotes a  probability space.
Proofs are presented at the end
of the section.

\begin{Theorem}
	\label{thm:saa:bOgutQ6IqE}
	Let $\hsp$ be a real, separable Hilbert space.
	Suppose that $\rv_i : \Theta\to \hsp$
	for $i = 1, 2, \ldots$ are independent, mean-zero random 
	variables such that
	$\cE{\exp(\tau^{-2}\norm[\hsp]{\rv_i}^2)}\leq\eu$ 
	for some $\tau > 0$. Then, for each
	$N \in \natural$, $\varepsilon \geq 0$, 
	\begin{align}
	\label{eq:saa:bOgutQ6IqE_1}
	\Prob{
		\norm[\hsp]{\rv_1 + \cdots + \rv_N}
		\geq N\varepsilon
	}
	\leq 
	2\exp(-\tau^{-2}\varepsilon^2N/3).
	\end{align}
\end{Theorem}
If in addition 
$\norm[\hsp]{\rv_i} \leq \tau$
\wpone\ for $i = 1, 2, \ldots$, then
the upper bound in \eqref{eq:saa:bOgutQ6IqE_1} improves to
$2\exp(-\tau^{-2}\varepsilon^2N/2)$
\cite[Thm.\ 3.5]{Pinelis1994}.	

As an alternative to 
the condition $\cE{\exp(\tau^{-2}\norm[\hsp]{\rv}^2)}\leq\eu$ 
used in \Cref{thm:saa:bOgutQ6IqE}
for $\tau > 0$ and 
a random vector $\rv : \Theta\to \hsp$, we can 
express sub-Gaussianity 
with $\cE{\cosh(\lambda \norm[\hsp]{\rv})}\leq \exp(\lambda^2\sigma^2/2)$ 
for  all $\lambda \in \real$ and some $\sigma > 0$.
While these two conditions are equivalent up to
problem-independent constants
(see the proof of \cite[Lem.\ 1.6 on p.\ 9]{Buldygin2000}
and \Cref{lem:saa:2020-01-21T19:41:52.108}),
the constant $\sigma$ can be smaller than $\tau$.
For example, if $\rv : \Theta\to \hsp$
is a $\hsp$-valued, mean-zero Gaussian random vector, then 
the latter condition holds with $\sigma^2 = \cE{\norm[\hsp]{\rv}^2}$
\cite[Rem.\ 4]{Pinelis1986}.
However,  if $\hsp = \real$ then
$	\tau^2 = 2\sigma^2/(1-\exp(-2)) 
\href{https://tinyurl.com/yymabckc}
{\approx 2.31}\sigma^2
$ 
\cite[p.\ 9]{Buldygin2000}.

\begin{proposition}
	\label{prop:saa:2020-11-21T20:25:01.71}
	Let $\hsp$ be a real, separable Hilbert space, and let
	$\rv_i : \Theta\to \hsp$
	be independent, mean-zero random 
	vectors such that
	$\cE{\cosh(\lambda \norm[\hsp]{\rv_i})}\leq \exp(\lambda^2\sigma^2/2)$ 
	for all $\lambda \in \real$ and
	some $\sigma > 0$  ($i = 1, 2, \ldots$). Then, for each
	$N \in \natural$, $\varepsilon \geq 0$, 
	$$
	\Prob{
		\norm[\hsp]{\rv_1 + \cdots + \rv_N}
		\geq N\varepsilon
	}
	\leq 
	2\exp(-\sigma^{-2}\varepsilon^2N/3).
	$$
\end{proposition}

We apply the following two facts to prove 
\Cref{thm:saa:bOgutQ6IqE,prop:saa:2020-11-21T20:25:01.71}.
\begin{Theorem}
	[{see \cite[Thm.\ 3]{Pinelis1986}}]
	\label{thm:saa:2020-03-15T00:12:53.514}
	Let $\hsp$ be a real, separable Hilbert space.
	Suppose that $\rv_i : \Theta\to \hsp$ 
	$(i=1, \ldots, N \in \natural)$ are independent, mean-zero 
	random vectors.
	Then, for all $\lambda \geq 0$, 
	$$
	\cE{\cosh(\lambda\norm[\hsp]{ \rv_1 + \cdots + \rv_N})} 
	\leq \prod_{i=1}^N 
	\cE{\exp(\lambda \norm[\hsp]{\rv_i})- \lambda \norm[\hsp]{\rv_i}}.
	$$
\end{Theorem}

\begin{lemma}
	\label{lem:saa:2020-01-21T19:41:52.108}
	If $\sigma > 0$ and	$\rvariable : \Theta
	\to \real$ is measurable with
	$\cE{\exp(\sigma^{-2}|\rvariable|^2)} \leq \eu$,
	\begin{align}
	\label{eq:saa:2020-01-19T11:52:16.326}
	\cE{\exp(\lambda |\rvariable|) - \lambda|\rvariable|}
	\leq \exp(3\lambda^2 \sigma^2/4) 
	\quad \tfa \quad 
	\lambda \in \real_+. 
	\end{align}
\end{lemma}
\begin{proof}
	The proof is based on the proof of \cite[Prop.\ 7.72]{Shapiro2014}.
	
	Fix $\lambda \in [0, 4/(3 \sigma)]$.
	For all $s \in \real$, 
	$\exp(s) \leq s + \exp(9s^2/16)$ \cite[p.\ 449]{Shapiro2014}.  
	Using Jensen's inequality 	and 
	$\cE{\exp(|\rvariable|^2/\sigma^2)} \leq \eu$, 
	we obtain
	\begin{align}
	\label{eq:saa:2020-01-19T12:25:22.097}
	\begin{aligned}
	\cE{\eu^{\lambda|\rvariable|} - \lambda |\rvariable|}
	\leq  \cE{\eu^{9\lambda^2 |\rvariable|^2/16}}
	\leq
	\cE{\eu^{|\rvariable|^2/\sigma^2}}^{9\lambda^2\sigma^2/16}
	\leq 
	\eu^{9 \lambda^2 \sigma^2/16}.
	\end{aligned}
	\end{align}
	
	Now, fix $\lambda  \geq 4/(3\sigma)$.
	For all $s \in \real$,
	Young's inequality yields
	$\lambda s \leq  3 \lambda^2 \sigma^2/8
	+ 2s^2/(3\sigma^2)$.
	Combined with
	Jensen's inequality, 
	$\cE{\exp(|\rvariable|^2/\sigma^2)} \leq \eu$,
	and $2/3 \leq 3\sigma^2\lambda^2/8$, 
	we get
	\begin{align*}
	\begin{aligned}
	\cE{\eu^{\lambda |\rvariable|} - \lambda |\rvariable|}
	& \leq 
	\cE{\eu^{\lambda |\rvariable|}}
	\leq 
	\eu^{3\lambda^2 \sigma^2/8}
	\cE{\eu^{2|\rvariable|^2/(3\sigma^2)}}
	\leq 
	\eu^{3\lambda^2 \sigma^2/8 + 2/3}
	\leq 
	\eu^{3\lambda^2 \sigma^2/4}.
	\end{aligned}
	\end{align*}
	Together with \eqref{eq:saa:2020-01-19T12:25:22.097}, 	we obtain
	\eqref{eq:saa:2020-01-19T11:52:16.326}. 
\end{proof}
\begin{proof}[{Proof of \Cref{thm:saa:bOgutQ6IqE}}]
	We use a Chernoff-type approach to 
	establish \eqref{eq:saa:bOgutQ6IqE_1}.
	Fix $\lambda > 0$, $\varepsilon \geq  0$, 
	and $N \in \natural$. We define
	$\randomsum_N = \rv_1 + \cdots + \rv_N$.
	Using $\cE{\exp(\tau^{-2}\norm[\hsp]{\rv_i}^2)}\leq\eu$
	and applying \Cref{lem:saa:2020-01-21T19:41:52.108} 
	to $\rvariable = \norm[\hsp]{\rv_i}$,
	we find that 
	\begin{align*}
	\prod_{i = 1}^N
	\cE{\exp(\lambda \norm[\hsp]{\rv_i}) -\lambda \norm[\hsp]{\rv_i}}
	\leq 
	\prod_{i = 1}^N  \exp(3\lambda^2 \tau^2/4) 
	= \exp(3\lambda^2 \tau^2 N/4).
	\end{align*}
	Combined with Markov's inequality, \Cref{thm:saa:2020-03-15T00:12:53.514}, 
	and $\exp \leq 2 \cosh$, we obtain
	\begin{align*} 
	\Prob{\norm[\hsp]{\randomsum_N} \geq N\varepsilon}
	&\leq \eu^{-\lambda N\varepsilon}
	\cE{\eu^{\lambda \norm[\hsp]{\randomsum_N}}}
	\leq 2\eu^{-\lambda N\varepsilon}
	\cE{\cosh(\lambda \norm[\hsp]{\randomsum_N})}
	\\
	&\leq 2\eu^{-\lambda N \varepsilon + 3\lambda^2 \tau^2 N/4}.
	\end{align*}
	Minimizing the right-hand side over $\lambda > 0$
	yields \eqref{eq:saa:bOgutQ6IqE_1}.
\end{proof}

\begin{proof}[{Proof of \Cref{prop:saa:2020-11-21T20:25:01.71}}]
	We have \href{https://tinyurl.com/y4xq8tne}
	{$\exp(s)-s \leq \cosh( s \sqrt{3/2})$}
	for all $s \in \real$.
	Hence, the assumptions ensure
	$
	\cE{\exp(\lambda \norm[\hsp]{\rv_i}) -\lambda \norm[\hsp]{\rv_i}}
	\leq \exp(3\lambda^2\sigma^2/4)
	$
	for all $\lambda \in \real$.
	The remainder of the proof is as that of 
	\Cref{thm:saa:bOgutQ6IqE}. 
\end{proof}

\section{Exponential Tail Bounds for SAA Solutions}
\label{sec:tailboundssaa}

We state conditions that allow
us to derive 
exponential bounds on the tail probabilities of
the distance between SAA solutions and their true counterparts.
In \cref{subsec:saa:lqocu}, we demonstrate that our
conditions are fulfilled for many linear-quadratic
control problems considered in the literature.

\subsection{Assumptions and Measurability of SAA Solutions}

Throughout the manuscript, 
$u^*$ is assumed to be a solution to \eqref{eq:saa:2020-01-15T21:11:01.962}.
\begin{assumption}
	\label{ass:2020-07-06T10:42:18.634_1}
	\begin{enumthm}[nosep,leftmargin=*]
		\item 
		\label{ass:2020-07-06T10:42:18.634_11}
		The space $\csp$ is a  real, separable Hilbert
		space.
		\item 
		\label{ass:2020-07-06T10:42:18.634_13}
		The function 
		$\Psi : \csp \to \real \cup \{\infty\}$ 
		is convex, proper, and lower-semicontinuous.
		\item 
		\label{ass:2020-07-06T10:42:18.634_22}
		The integrand
		$\rpobj : \csp \times \Xi \to \real$ 
		is a \Caratheodory\ function, and
		for some $\alpha  >0$,
		$\rpobj(\cdot, \xi)$ is $\alpha$-strongly convex
		for each $\xi \in \Xi$.
		\item 
		\label{ass:2020-07-06T10:42:18.634_21}
		The function
		$\rpobj(\cdot, \xi)$ is 	
		\gateaux\ differentiable
		on a convex neighborhood of $\dom{\Psi}$
		for all $\xi \in \Xi$, 
		and $\nabla_u \rpobj(u^*, \cdot)  : \Xi \to \csp$
		is measurable.
		\item 
		\label{ass:2020-07-06T10:42:18.634_12}
		The map
		$\erpobj : \csp \to \real \cup \{\infty\}$ 
		defined in \eqref{eq:saa:erpobjN} is \gateaux\ differentiable
		at $u^*$.
	\end{enumthm}
\end{assumption}%
\begin{lemma}
	\label{lem:saa:2020-07-08T14:08:44.744_1}
	Let 
	Assumptions~\ref{ass:2020-07-06T10:42:18.634_11}--\ref{ass:2020-07-06T10:42:18.634_22}
	hold.
	If $u_N^* : \Omega \to \csp$  is a solution
	to \eqref{eq:saa:saalqcp}, 
	then $u_N^*$
	is the unique  solution to \eqref{eq:saa:saalqcp}
	and is measurable.
\end{lemma}
\begin{proof}
	For each $\omega \in \Omega$,
	the SAA problem's objective function $\obj_N(\cdot, \omega)$
	is strongly convex and hence 
	$u_N^*$ is the unique solution to 
	\eqref{eq:saa:saalqcp}.
	The function 
	$\inf_{u \in \csp}\, \obj_N(u,\cdot) : \pOmega \to \real$
	is measurable
	\cite[Cor.\ VII-2]{Castaing1977}
	(see also \cite[Lem.\ III.39]{Castaing1977}).
	Hence the multifunction $\arg\inf_{u \in \csp}\, \obj_N(u, \cdot)$ 
	is single-valued and has a measurable selection 
	\cite[Thm.\ 8.2.9]{Aubin2009}.
	Therefore 	$u_N^* : \pOmega \to \csp$ is measurable.
\end{proof}

We impose conditions on the integrability of
$\nabla_u \rpobj(u^*, \xi)-\nabla \erpobj(u^*)$.
\begin{assumption}
	\label{ass:at}
	\begin{enumthm}[nosep,leftmargin=*]
		\item 
		\label{ass:2020-07-06T10:42:18.634_4}
		For some $\sigma > 0$, 
		$\cE{\norm[\csp]{\nabla_u \rpobj(u^*, \xi)-
				\nabla \erpobj(u^*)}^2} \leq \sigma^2$.
		\item 
		\label{ass:2020-07-06T10:42:18.634_3}
		For some $\tau > 0$, 
		$
		\cE{\exp 
			(
			\tau^{-2}\norm[\csp]{\nabla_u \rpobj(u^*, \xi)-
				\nabla \erpobj(u^*)}^2
			)
		}
		\leq \eu
		$.
	\end{enumthm}
\end{assumption}
\Cref{ass:2020-07-06T10:42:18.634_3}
implies
\Cref{ass:2020-07-06T10:42:18.634_4} with $\sigma^2 = \tau^2$
\cite[p.\ 1584]{Nemirovski2009}.
\Cref{ass:2020-07-06T10:42:18.634_3} and its variants
are standard conditions in the literature on stochastic programming 
\cite[p.\ 679]{Duchi2012}, 
\cite[pp.\ 1035--1036]{Guigues2017}, 
\cite[eq.\ (2.50)]{Nemirovski2009}.
For example, if $\nabla_u \rpobj(u^*, \xi)-\nabla \erpobj(u^*)$
is essentially bounded, then \Cref{ass:2020-07-06T10:42:18.634_3} is fulfilled.
More generally, if $\nabla_u \rpobj(u^*, \xi)-\nabla \erpobj(u^*)$
is $\gamma$-sub-Gaussian, then \Cref{ass:2020-07-06T10:42:18.634_3} 
holds true \cite[Thm.\ 3.4]{Fukuda1990}.

\subsection{Exponential Tail and Mean Square Error  Bounds}
We establish exponential tail and mean square error bounds
on $\norm[\csp]{u^*-u_N^*}$.
\begin{Theorem}
	\label{thm:saa:2020-01-25T19:33:40.055}
	Let $u^*$ be a solution to \eqref{eq:saa:2020-01-15T21:11:01.962}
	and let $u_N^*$
	be a solution to \eqref{eq:saa:saalqcp}.
	If Assumptions~\ref{ass:2020-07-06T10:42:18.634_1} and
	\ref{ass:2020-07-06T10:42:18.634_4} hold, 
	then 
	\begin{align}
	\label{eq:Feb2820211612}
	\cE{\norm[\csp]{u^*-u_N^*}^2} \leq \sigma^2/(\alpha^2 N).
	\end{align}
	If in addition 	\Cref{ass:2020-07-06T10:42:18.634_3}
	holds, then
	for all $\varepsilon > 0$, 
	\begin{align}
	\label{eq:saa:2020-01-25T19:33:40.055}
	\Prob{\norm[\csp]{u^*-u_N^*} \geq \varepsilon}
	\leq 
	2\exp(-\tau^{-2}N \varepsilon^2 \alpha^2/3).
	\end{align}
\end{Theorem}

We prepare our proof of \Cref{thm:saa:2020-01-25T19:33:40.055}.

\begin{lemma}
	\label{lem:saa:2020-04-01T18:25:27.402}
	If Assumptions~\ref{ass:2020-07-06T10:42:18.634_1}
	and 
	\ref{ass:2020-07-06T10:42:18.634_4}
	hold, then
	$\cE{\nabla_u\rpobj(u^*, \xi)} = \nabla \erpobj(u^*)$.
\end{lemma}
\begin{proof}
	Using
	\Cref{ass:2020-07-06T10:42:18.634_11,ass:2020-07-06T10:42:18.634_22,%
		ass:2020-07-06T10:42:18.634_21,ass:2020-07-06T10:42:18.634_12}, 
	we have
	$\cE{\inner[\csp]{\nabla_u \rpobj(u^*, \xi)}{v}} = 
	\inner[\csp]{\nabla \erpobj(u^*)}{v}$
	for all $v \in \csp$; cf.\
	\cite[p.\ 1050]{Guigues2017}.
	Owing to Assumptions~\ref{ass:2020-07-06T10:42:18.634_12}
	and \ref{ass:2020-07-06T10:42:18.634_4},
	the mapping $\nabla_u \rpobj(u^*, \xi)$ is 
	integrable.
	Hence 
	$\cE{\inner[\csp]{\nabla_u \rpobj(u^*, \xi)}{v}} 
	= 
	\inner[\csp]{\cE{\nabla_u \rpobj(u^*, \xi)}}{v}
	$
	for all $v \in \csp$
	(cf.\ \cite[p.\ 78]{Bharucha-Reid1972}).
\end{proof}

\begin{lemma}
	\label{lem:saa:2020-01-25T20:42:37.498}
	If Assumption~\ref{ass:2020-07-06T10:42:18.634_1} holds, then
	the function $\erpobj_N$
	defined in \eqref{eq:saa:erpobjN} 
	is \gateaux\ differentiable on a  neighborhood of $\dom{\Psi}$ and
	\wpone, 
	\begin{align}
	\label{eq:saa:2020-01-25T20:42:37.498}
	\innerp{\csp}{\nabla \erpobj_N(u_2) - \nabla \erpobj_N(u_1)}{u_2 - u_1}
	\geq \alpha \norm[U]{u_2-u_1}^2
	\,\,\, \tfa \,\,\, u_1, u_2 \in \dom{\Psi}. 
	\end{align}
\end{lemma}
\begin{proof}
	Since, for each $\xi \in \Xi$, 
	$\rpobj(\cdot, \xi)$ is $\alpha$-strongly convex
	and \gateaux\ differentiable on a convex neighborhood
	$\bsp$ of $\dom{\Psi}$, 
	the sum rule and the definition of 
	$\erpobj_N$  imply its
	\gateaux\ differentiability
	on $\bsp$ and \eqref{eq:saa:2020-01-25T20:42:37.498}
	\cite[p.\ 48]{Nemirovsky1983}.
\end{proof}
\begin{lemma}
	\label{lem:saa:fonoclqcp}
	Let \Cref{ass:2020-07-06T10:42:18.634_1} hold
	and let $\omega \in \pOmega$ be fixed.
	Suppose that $u^*$ 
	is a solution to
	\eqref{eq:saa:2020-01-15T21:11:01.962} 
	and that $u_N^* = u_N^*(\omega)$ is a solution to
	\eqref{eq:saa:saalqcp}.
	Then 
	\begin{align}
	\label{eq:Feb2820211042}
	\innerp{\csp}{\nabla \erpobj_N(u_N^*)- \nabla \erpobj(u^*)}{u^* - u_N^*}
	\geq 0.
	\end{align}
\end{lemma}
\begin{proof}
	Following the proof of \cite[Thm.\ 4.42]{Ito2008},
	we obtain for all $u \in \dom{\Psi}$,
	\begin{align}
	\label{eq:saa:2020-01-15T21:18:29.76}
	\begin{aligned}
	\innerp{\csp}{\nabla \erpobj(u^*)}{u- u^*}  + \Psi(u) - \Psi(u^*) &\geq 0,
	\\
	\innerp{\csp}
	{\nabla \erpobj_N(u_N^*)}{u- u_N^*}  
	+ \Psi(u) - \Psi(u_N^*) 
	&\geq 0.
	\end{aligned}
	\end{align}
	We have  $\Psi(u^*)$, $\Psi(u_N^*)  \in \real$. 	
	Choosing $u = u_N^*$ in the first and  $u = u^*$ in 
	the second estimate
	in \eqref{eq:saa:2020-01-15T21:18:29.76},
	and adding the resulting inequalities yields 
	\eqref{eq:Feb2820211042}.
\end{proof}
\begin{lemma}
	\label{lem:saa:2020-01-25T19:33:12.799}
	Under the hypotheses of \Cref{lem:saa:fonoclqcp}, we have
	\begin{align}
	\label{eq:saa:2020-01-25T19:33:12.799}
	\alpha \norm[\csp]{u^*-u_N^*}
	\leq \cnorm{\csp}{\nabla \erpobj_N(u^*)- \nabla \erpobj(u^*)}.
	\end{align}
\end{lemma}
\begin{proof}
	Choosing $u_2 = u^*$ and $u_1 = u_N^*$ in
	\eqref{eq:saa:2020-01-25T20:42:37.498}, 
	we find that
	\begin{align*}
	\innerp{\csp}{\nabla \erpobj_N(u^*) - \nabla \erpobj_N(u_N^*)}{u^* - u_N^*}
	\geq \alpha \norm[\csp]{u^*-u_N^*}^2. 
	\end{align*}
	Combined with \eqref{eq:Feb2820211042},
	and
	the Cauchy--Schwarz inequality, we get
	\begin{align*}
	\alpha \norm[\csp]{u^*-u_N^*}^2
	& \leq 
	\innerp{\csp}{\nabla \erpobj_N(u^*) - \nabla \erpobj_N(u_N^*)}
	{u^* - u_N^*}
	\\ &\quad  +
	\innerp{\csp}{\nabla \erpobj_N(u_N^*)- \nabla \erpobj(u^*)}{u^* - u_N^*}
	\\
	& \leq \cnorm{\csp}{\nabla \erpobj_N(u^*)- \nabla \erpobj(u^*)}
	\cnorm{\csp}{u^* - u_N^*}.
	\end{align*}
\end{proof}

\begin{proof}[{Proof of \Cref{thm:saa:2020-01-25T19:33:40.055}}] 
	\Cref{lem:saa:2020-07-08T14:08:44.744_1} ensures
	the measurability of 
	$u_N^* : \pOmega \to \csp$.
	We define $q : \Xi \to \csp$
	by $q(\xi) = \nabla_u \rpobj(u^*, \xi)-\nabla \erpobj(u^*)$.
	\Cref{ass:2020-07-06T10:42:18.634_12,ass:2020-07-06T10:42:18.634_22} 
	ensure that
	$q$ is well-defined and
	measurable. Hence, the random vectors $Z_i = q(\xi^i)$
	($i = 1, 2,  \ldots$) are independent identically distributed,
	and \Cref{lem:saa:2020-04-01T18:25:27.402}
	ensures that they have zero mean.
	Using the definitions
	of $\erpobj$ and of $\erpobj_N$
	provided in 
	\eqref{eq:saa:erpobjN}, 
	the \gateaux\ differentiability of $\erpobj$
	at $u^*$ (see \Cref{ass:2020-07-06T10:42:18.634_12}), 
	and 
	\Cref{lem:saa:2020-01-25T20:42:37.498}, we obtain
	$$
	\nabla \erpobj_N(u^*)-\nabla \erpobj(u^*)
	= \frac{1}{N}
	\sum_{i=1}^N \big(\nabla_u \rpobj(u^*, \xi^i)-\nabla \erpobj(u^*)\big)
	= \frac{1}{N} \sum_{i=1}^N Z_i.
	$$

	Now, we prove \eqref{eq:Feb2820211612}.
	Combining the above statements
	with the separability of the Hilbert space $\csp$,
	we get
	$\cE{\norm[\csp]{\sum_{i=1}^N Z_i}^2}
	= \sum_{i=1}^N \cE{\norm[\csp]{Z_i}^2}
	$ \cite[p.\ 79]{Yurinsky1995}.
	For $i=1, 2, \ldots$, 
	\Cref{ass:2020-07-06T10:42:18.634_4}
	yields
	$\cE{\norm[\csp]{Z_i}^2} \leq \sigma^2$.
	Together with the estimate
	\eqref{eq:saa:2020-01-25T19:33:12.799},
	we find that
	$$
	\alpha^2 
	\cE{ \norm[\csp]{u^*-u_N^*}^2}
	\leq 
	\cE{\cnorm{\csp}{\nabla \erpobj_N(u^*)- \nabla \erpobj(u^*)}^2}
	\leq 
	\sigma^2/N,
	$$
	yielding the mean square error bound \eqref{eq:Feb2820211612}.

	Next, we establish \eqref{eq:saa:2020-01-25T19:33:40.055}.
	Fix $\varepsilon > 0$. If
	$\norm[\csp]{u^*-u_N^*} \geq \varepsilon$, then
	the estimate \eqref{eq:saa:2020-01-25T19:33:12.799} ensures that
	$\norm[\csp]{\sum_{i=1}^N Z_i} \geq N\alpha\varepsilon$.
	For $i = 1, 2, \ldots$,
	\Cref{ass:2020-07-06T10:42:18.634_3}
	implies that
	$
	\cE{\exp 
		(
		\tau^{-2}\norm[\csp]{Z_i}^2
		)
	}
	\leq \eu
	$.
	Applying \Cref{thm:saa:bOgutQ6IqE}, we get
	\begin{align*}
	\Prob{\norm[\csp]{u^*-u_N^*} \geq \varepsilon}
	\leq 
	\Prob[\bigg]{\cnorm[\Big]{\csp}
		{\sum_{i=1}^N Z_i} \geq N\alpha\varepsilon}
	\leq 2\eu^{-\tau^{-2}\varepsilon^2\alpha^2 N/3}.
	\end{align*}
	Hence the exponential 
	tail bound \eqref{eq:saa:2020-01-25T19:33:40.055} holds true.
\end{proof}

\section{Optimality of SAA Solutions' Exponential Tail Bounds}
\label{sec:optimalitytailbounds}

We show that the dependence of the tail bound 
\eqref{eq:saa:2020-01-25T19:33:40.055} on
the problem data is essentially optimal
for the problem class modeled by
Assumptions~\ref{ass:2020-07-06T10:42:18.634_1} 
and
\ref{ass:2020-07-06T10:42:18.634_3}.

Our example is inspired by that 
analyzed in \cite[Ex.\ 1]{Shapiro2008}.
We consider
\begin{align}
\label{eq:saa:2020-10-24T13:44:59.97}
\min_{u\in L^2(0,1)}\, 
\cE{(\alpha/2)\norm[L^2(0,1)]{u}^2
	- \inner[L^2(0,1)]{\rrhs(\xi)}{u}},
\end{align}
where $\alpha > 0$,
$\varphi_1$, $\varphi_2 \in L^2(0,1)$
are orthonormal,
$\rrhs : \real^2 \to L^2(0,1)$
is given by
$\rrhs(\xi) = \xi_1 \varphi_1 + \xi_2 \varphi_2$,
and  $\xi_1$, $\xi_2$ are independent, 
standard Gaussians.
The solution $u^*$ to 
\eqref{eq:saa:2020-10-24T13:44:59.97} is $u^* = 0$
since $\cE{\rrhs(\xi)} = 0$, and
the SAA solution $u_N^*$
corresponding to \eqref{eq:saa:2020-10-24T13:44:59.97}	is
$u_N^* = (1/\alpha)\bar{\xi}_{1,N}\varphi_1 + (1/\alpha)\bar{\xi}_{2,N}\varphi_2$,
where $\bar{\xi}_{j,N} = (1/N)\sum_{i=1}^N \xi_j^i$
for $j =1, 2$.
The orthonormality of $\varphi_1$, $\varphi_2$
yields
$
\norm[L^2(0,1)]{u_N^*}^2 =  
(1/\alpha)^2 (\bar{\xi}_{1,N})^2 +  (1/\alpha)^2(\bar{\xi}_{2,N})^2
$.
Since $(1/\alpha)\bar{\xi}_{1,N}$ and 
$(1/\alpha)\bar{\xi}_{2,N}$ are independent, 
mean-zero Gaussian with variance $N^{-1}\alpha^{-2}$,
the random variable
$N\alpha^2\norm[L^2(0,1)]{u^*-u_N^*}$ 
has a chi-square distribution
$\chi_2^2$ with two degrees of freedom.
Hence, for all $\varepsilon \geq 0$,
\begin{align}
\label{eq:saa:2020-10-24T16:00:24.521}
\Prob{\norm[L^2(0,1)]{u^*-u_N^*} \geq \varepsilon}
=
\Prob{\chi_2^2 \geq N\alpha^2\varepsilon^2}
= 
\eu^{-N\alpha^2\varepsilon^2/2}.
\end{align} 
Since 
$\rpobj(u, \xi) = (\alpha/2) \norm[L^2(0,1)]{u}^2
+\inner[L^2(0,1)]{\rrhs(\xi)}{u}$
and $\erpobj(u) = (\alpha/2)\norm[L^2(0,1)]{u}^2$,
we find that
$
\norm[L^2(0,1)]{\nabla_u \rpobj(u^*, \xi)-\nabla\erpobj(u^*)}^2
= \norm[L^2(0,1)]{\rrhs(\xi)}^2
$.
Combined with $\norm[L^2(0,1)]{\rrhs(\xi)}^2 \sim \chi_2^2$, 
we obtain
\href{https://tinyurl.com/nehw627d}
{$\cE{\exp(\tau^{-2}\norm[L^2(0,1)]{\rrhs(\xi)}^2)} = \eu$
	for $\tau^2=2\eu/(\eu-1)$}.
Our computations
and the tail bound \eqref{eq:saa:2020-10-24T16:00:24.521} reveal that the 
exponential order of the tail bound in
\eqref{eq:saa:2020-01-25T19:33:40.055}
is optimal up to the 	constant 
$3\tau^2/2 
\href{https://tinyurl.com/d3f79fyj}{\approx 4.7}
$.

\section{Application to Linear-Quadratic Optimal Control}
\label{subsec:saa:lqocu}

We consider the linear-quadratic optimal control problem 
\begin{align}
\label{eq:saa:lqcp}
\min_{u \in \csp } \, 
\{\, 
(1/2)\cE{\norm[\hsp]{QS(u, \xi)-y_d}^2} + (\alpha/2)\norm[\csp]{u}^2
+ \Psi(u)
\, \},
\end{align}
where $\alpha > 0$,  $Q \in  \sblf[\hsp]{\ssp}$,
$y_d \in \hsp$ and $\hsp$ is a real, separable Hilbert space.
In this section,
$\csp$ and $\Psi : \csp \to \real \cup \{\infty\}$
fulfill 
Assumptions~\ref{ass:2020-07-06T10:42:18.634_11}
and
\ref{ass:2020-07-06T10:42:18.634_13}, 
respectively.
The parameterized solution operator
$S : \csp \times \Xi \to \ssp$ is defined as follows.
For each $(u, \xi) \in \csp \times \Xi$, $S(u, \xi)$ is the solution to:
\begin{align}
\label{eq:saa:leq}
\text{find} \quad y \in \ssp: \quad A(\xi)y + B(\xi) u = g(\xi).
\end{align}
The spaces $\ssp$ and 
$\asp$ are  real, separable Banach spaces, 
$A : \Xi \to \sblf[\asp]{\ssp}$ and
$B : \Xi \to \sblf[\asp]{\csp}$, 
$A(\xi)$ has a bounded inverse for each $\xi \in \Xi$,
and
$g : \Xi  \to \asp$.

We can model 
parameterized affine-linear elliptic and parabolic
PDEs with \eqref{eq:saa:leq}, such as the heat equation
with random inputs considered in \cite[Sect.\ 3.1.2]{MartnezFrutos2018}, 
and the elliptic PDEs with random inputs considered 
\cite{Martin2021,Garreis2019b,Sun2015}.
When $\domain \subset \real^d$ is
a bounded domain and $\csp = L^2(\domain)$, a popular choice has been 
$\Psi(\cdot)
= \gamma \norm[L^1(\domain)]{\cdot} + I_{\adcsp}(\cdot)$
for $\gamma \geq 0$, 
where $\adcsp \subset \csp$ is a nonempty, convex, closed set
\cite{Geiersbach2020,Stadler2009}.
Further nonsmooth
regularizers are considered in \cite[Sect.\ 4.7]{Ito2008}.

Defining $	K(\xi) = -QA(\xi)^{-1}B(\xi)$
and $h(\xi) = QA^{-1}(\xi)g(\xi)-y_d$,
the control problem
\eqref{eq:saa:lqcp} can be written as
\begin{align}
\label{eq:inverse}
\min_{u \in \csp } \, 
\{\, 
(1/2)\cE{\norm[\hsp]{K(\xi)u+h(\xi)}^2} +
(\alpha/2)\norm[\csp]{u}^2
+ \Psi(u)
\, \}.
\end{align}
We discuss differentiablity and the lack of 
strong convexity of the expectation function
$\erpobj_1 : \csp \to \real \cup \{\infty\}$ defined by
\begin{align}
\label{eq:erpobj1}
\erpobj_1(u) = (1/2)\cE{\norm[\hsp]{K(\xi)u + h(\xi)}^2}.
\end{align}

\begin{assumption}
	\label{ass:2020-07-08T12:04:08.104}
	The map $K : \Xi \to \sblf[\hsp]{\csp}$
	is strongly measurable and
	$h : \Xi \to \hsp$ is strongly measurable.
	For each $u \in \csp$, 
	$\cE{\norm[\csp]{K(\xi)^*K(\xi)u}} < \infty$, and
	$\cE{\norm[\hsp]{h(\xi)}^2}$,
	$\cE{\norm[\csp]{K(\xi)^*h(\xi)}} < \infty$.
\end{assumption}
We define the integrand $\rpobj_1 : \csp \times \Xi \to \real$ by
\begin{align}
\label{eq:rpobj1}
\rpobj_1(u, \xi) = (1/2)\norm[\hsp]{K(\xi)u+h(\xi)}^2.
\end{align}
Under the measurability conditions 
stated in \Cref{ass:2020-07-08T12:04:08.104}, we can show that
$\rpobj_1$ is a \Caratheodory\
function.

\Cref{ass:2020-07-08T12:04:08.104} 
implies that the function $\erpobj_1$ defined in
\eqref{eq:erpobj1} is smooth.

\begin{lemma}
	\label{lem:saa:gradhesserpobj1}
	If \Cref{ass:2020-07-08T12:04:08.104} holds,
	then $\erpobj_1$ 
	defined in  \eqref{eq:erpobj1} 
	is infinitely many times continuously differentiable,
	and for all $u$, $v \in \csp$,
	\begin{align*}
	\nabla \erpobj_1(u) = \cE{K(\xi)^*(K(\xi)u+h(\xi))}
	\quad \tand \quad 
	\nabla^2 \erpobj_1(u)[v] = 
	\cE{K(\xi)^*K(\xi)v}.
	\end{align*}
\end{lemma}
\begin{proof}
	The strong measurability of $K$
	implies that of $\xi \mapsto K(\xi)^*$
	\cite[Thm.\ 1.1.6]{Hytoenen2016}
	and hence that of $\xi \mapsto K(\xi)^* K(\xi)$
	\cite[Cor.\ 1.1.29]{Hytoenen2016}.
	Fix $u$, $ v\in \csp$ and $\xi \in \Xi$.
	Since
	$
	\norm[\hsp]{K(\xi)u}^2 \leq \norm[\csp]{u}\norm[\csp]{K(\xi)^*K(\xi)u}
	$
	\cite[p.\ 199]{Kreyszig1978},
	\Cref{ass:2020-07-08T12:04:08.104} 
	ensures that
	$\erpobj_1$ is finite-valued. 
	Using \eqref{eq:rpobj1}, we 
	find that
	$\nabla_u \rpobj_1(u, \xi) = K(\xi)^*(K(\xi)u+h(\xi))$
	and 
	$\nabla_{uu} \rpobj_1(u, \xi)[v] = K(\xi)^*K(\xi)v$.
	Combined with \Cref{ass:2020-07-08T12:04:08.104}
	and \cite[Lem.\ C.3]{Geiersbach2020}, 
	we obtain that $\erpobj_1$ is \gateaux\ differentiable 
	with $\nabla \erpobj_1(u) = \cE{\nabla_u \rpobj_1(u, \xi)}$.
	Since $\cE{\norm[\csp]{K(\xi)^*K(\xi)w}} < \infty$
	for all $w\in \csp$,
	$w\mapsto \cE{K(\xi)^*K(\xi)w}$
	is linear and bounded 
	\cite[Thm.\ 3.8.2]{Hille1957}.
	Combined with the fact that $\rpobj_1(\cdot, \xi)$
	is quadratic for all $\xi \in \Xi$, 
	we conclude that $\erpobj_1$ is twice \gateaux\ differentiable
	with
	$
	\nabla^2 \erpobj_1(u)[v] = 
	\cE{K(\xi)^*K(\xi)v}
	$
	and hence 
	infinitely many times continuously differentiable.
\end{proof}

The function
$\erpobj_1$ 
defined in \eqref{eq:erpobj1} 
lacks strong convexity under natural conditions;
see \Cref{lem:lsc_erpobj}.
In this case, we may deduce 
that the strong convexity of the objective function of \eqref{eq:saa:lqcp}
solely comes from the function
$(\alpha/2)\norm[\csp]{\cdot}^2 + \Psi(\cdot)$, 
and that the largest strong convexity parameter of
$\erpobj(\cdot) = \erpobj_1(\cdot)+ (\alpha/2)\norm[\csp]{\cdot}^2$
is $\alpha > 0$.

\begin{assumption}
	\label{ass:2021-04-02T11:40:52.142}
	The mapping $K : \Xi \to \sblf[\hsp]{\csp}$
	is uniformly measurable, 
	$\cE{\norm[{\sblf[\hsp]{\csp}}]{K(\xi)}^2} < \infty$,
	and $K(\xi)$ is compact for all $\xi \in \Xi$. 
	Moreover, the Hilbert space $\csp$ is infinite-dimensional.
\end{assumption}

\begin{lemma}
	\label{lem:lsc_erpobj}
	If \Cref{ass:2020-07-08T12:04:08.104,ass:2021-04-02T11:40:52.142}
	hold,
	then  the expectation function $\erpobj_1$ 
	defined in \eqref{eq:erpobj1} is not strongly convex.
\end{lemma}

\begin{proof}
	We define $T : \Xi \to \sblf[\csp]{\csp}$
	by $T(\xi) = K(\xi)^*K(\xi)$.
	The uniform measurability of $K$ 
	implies that of $\xi \mapsto K(\xi)^*$
	(cf.\ \cite[Thm.\ 2.16]{Bharucha-Reid1972} and
	\cite[p.\ 200]{Kreyszig1978})
	and hence that of $T$ (cf.\	\cite[pp.\ 12--13]{Hytoenen2016}).
	Since $K(\xi)$ is compact, 
	$T(\xi)$ is compact \cite[p.\ 427]{Kreyszig1978}.
	Moreover, we have
	$\cE{\norm[{\sblf[\csp]{\csp}}]{T(\xi)}} =
	\cE{\norm[{\sblf[\hsp]{\csp}}]{K(\xi)}^2}$
	\cite[Thm.\ 3.9-4]{Kreyszig1978}.

	We show that $\cE{T(\xi)}$ is a compact operator.
	Let $(v_k) \subset \csp$ be weakly converging to some $\bar{v} \in \csp$.
	Hence there exists $C \in (0,\infty)$ with
	$\norm[\csp]{v_k}\leq C$ for all $k \in \natural$
	\cite[Thm.\ 4.8-3]{Kreyszig1978}
	which implies
	$\norm[\csp]{T(\xi)v_k} \leq C\norm[{\sblf[\csp]{\csp}}]{T(\xi)}$
	for each $\xi \in \Xi$ and $k \in \natural$.
	Since $T(\xi)$ is compact for all $\xi \in \Xi$, we have
	for each $\xi \in \Xi$,
	$T(\xi) v_k \to T(\xi)\bar{v}$ as $k \to \infty$
	\cite[Prop.\ 3.3.3]{Conway1990}.
	Combined with $\cE{\norm[{\sblf[\csp]{\csp}}]{T(\xi)}} < \infty$,
	the dominated convergence theorem 
	\cite[Prop.\ 1.2.5]{Hytoenen2016}
	yields 
	$\cE{T(\xi)v_k} \to \cE{T(\xi)\bar{v}}$
	as $k \to \infty$. We also have
	$\cE{T(\xi)w} = \cE{T(\xi)}w$ 
	for all $w \in \csp$
	\cite[p.\ 85]{Hille1957}.
	Thus
	$\cE{T(\xi)}v_k \to \cE{T(\xi)}\bar{v}$
	as $k \to \infty$.
	Since $\csp$ is reflexive and
	$(v_k)$ is arbitrary, $\cE{T(\xi)}$
	is compact \cite[Prop.\ 3.3.3]{Conway1990}.

	Now, we show that $\erpobj_1$ is not strongly convex.
	Since $\csp$ is infinite-dimensional,
	the self-adjoint, compact operator
	$\cE{T(\xi)}$ lacks a bounded inverse \cite[p.\ 428]{Kreyszig1978}, 
	\cite[Thm.\ 3.8.1]{Hille1957}.
	Hence it is noncoercive
	\cite[Lem.\ 4.123]{Bonnans2013}.
	Combined with $\nabla^2\erpobj_1(0) = \cE{T(\xi)}$
	(see \Cref{lem:saa:gradhesserpobj1} and \cite[p.\ 85]{Hille1957}),
	we conclude that $\erpobj_1$ is not strongly convex.
\end{proof}

The compactness of the Hessian 
of $\erpobj_1$ may also be studied
using the theory on spectral decomposition of
compact, self-adjoint operators
\cite[p.\ 159]{Vakhania1987}, 
or the results on the compactness of covariance
operators \cite[p.\ 174]{Vakhania1987}.

\subsection{Examples}
\label{subsec:saa:applications}

Many instances of the linear-quadratic control 
problem \eqref{eq:saa:lqcp}
frequently encountered in the literature are defined by
the following data: $\alpha > 0$, 
$\hsp = \csp$, 
$\ssp$ is a real Hilbert space, 
$Q \in \sblf[\hsp]{\ssp}$
is the embedding operator of 
the compact embedding $\ssp \embedding \hsp$, 
$B \in  \sblf[\ssp^*]{\csp}$ and
$g : \Xi \to \ssp^*$ is essentially bounded.
Moreover 
$A : \Xi \to  \sblf[\ssp^*]{\ssp}$
is uniformly measurable and
there exist constants
$0 < \rdc_{\min}^* \leq \rdc_{\max}^* < \infty$ 
with
$\norm[{\sblf[\ssp^*]{\ssp}}]{A(\xi)} \leq \rdc_{\max}^*$
and
$\dualp[\ssp]{A(\xi)y}{y} \geq \rdc_{\min}^*\norm[\ssp]{y}^2$
for all $(y,\xi) \in \ssp \times \Xi$.
The conditions imply that
$A(\xi)$ has a bounded inverse for each $\xi \in \Xi$
\cite[p.\ 101]{Kreyszig1978}
and imply the existence of a solution
to \eqref{eq:saa:lqcp}
when combined with Fatou's lemma; cf.\ \cite[Thm.\ 1]{Hoffhues2020}. Moreover 
\Cref{ass:2020-07-06T10:42:18.634_1,ass:at,%
	ass:2020-07-08T12:04:08.104,ass:2021-04-02T11:40:52.142}
hold true.

We show that \Cref{ass:2020-07-06T10:42:18.634_3} 
is violated for the class of optimal control problems
where the operator $A$ is elliptic and defined by a log-normal
random diffusion coefficient \cite{Ali2017,Charrier2012}.
Let $Q$ and $B$ be the embedding operators
of the embeddings 
$H_0^1(0, 1) \embedding L^2(0,1)$
and  $L^2(0,1) \embedding H^{-1}(0,1)$, respectively.
We choose $\Psi = 0$, $ \csp =  L^2(0,1)$, 
$y_d(\cdot) = \sin(\pi\cdot)/\pi^2$, 
and $A(\xi) = \eu^{-\xi}\bar{A}$, 
where  the weak Laplacian operator $\bar{A}$ is defined by
$\dualpHzeroone[0,1]{\bar{A}y}{v} = \innerp{L^2(0,1)}{y'}{v'}$,
and $\xi$ is a standard Gaussian random variable.
We have $\cE{\eu^{2\xi}} = \eu^2$ and $\cE{\eu^{\xi}} = \eu^{1/2}$.
Since $(\pi^2, y_d)$ is an eigenpair of $\bar{A}$, 
we find that  
$u^* = -\pi^2\eu^{1/2}y_d/(\eu^2+\pi^4\alpha) $
satisfies the  sufficient optimality
condition of \eqref{eq:saa:lqcp},
the normal equation
$
\alpha u^*  + \cE{\eu^{2\xi}}\bar{K}^*
\bar{K} u^* =\cE{\eu^{\xi}}
\bar{K}^*y_d
$,
where $\bar{K} = -Q \bar{A}^{-1}B$.
Hence $u^*$ is the solution to \eqref{eq:saa:lqcp}
for the above data.
Using the definition of $\rpobj_1$ provided in \eqref{eq:rpobj1},
we obtain
\begin{align*}
\nabla_u\rpobj_1(u^*, \xi) = 
\frac{\eu^\xi y_d}{\pi^2}
-\frac{\eu^{1/2+2\xi}y_d}{\pi^{2}(\eu^2+\pi^4\alpha)}
\end{align*}
For each $\xi \geq \ln(2(\eu^2+\pi^4\alpha))-1/2$, 
$\norm[L^2(0,1)]{\nabla_u\rpobj_1(u^*, \xi)} \geq (\eu^{\xi}/\pi^2)
\norm[L^2(0,1)]{y_d}$.
Combined with 
$
\nabla_u \rpobj(u^*,\xi) - \nabla \erpobj(u^*)
= 
\nabla_u \rpobj_1(u^*,\xi) - \nabla\erpobj_1(u^*)
$, 
$y_d \in L^2(0,1)$, 
and $\cE{\exp(s\xi^2/2)} = \infty$ for all $s \geq 1$
\cite[p.\ 9]{Buldygin2000}, we conclude that
\Cref{ass:2020-07-06T10:42:18.634_3} is violated.

\section{Numerical Illustration}
\label{sec:numresults}

We empirically verify the results 
derived in \Cref{thm:saa:2020-01-25T19:33:40.055}
for finite element
discretizations of two linear-quadratic, elliptic optimal control
problems, which are instances of \eqref{eq:saa:lqcp}.

For both examples, we consider
$\domain = (0,1)^2$,
and the mapping $Q$ in \eqref{eq:saa:lqcp} is the embedding operator of
the compact embedding $H_0^1(\domain) \embedding L^2(\domain)$.
Moreover, we define $y_d \in L^2(\domain)$ by
$y_d(x_1, x_2) = (1/6)\exp(2x_1)\sin(2\pi x_1)\sin(2\pi x_2)$
as in \cite[p.\ 511]{Bergounioux2000}.
For each $(u,\xi) \in L^2(\domain) \times \Xi$, 
$y(\xi) = S(u,\xi) \in H_0^1(\domain)$
solves the weak form of the linear elliptic PDE
\begin{align*}
-\nabla \cdot (\kappa(x,\xi)\nabla y(x,\xi)) = u(x) + r(x,\xi),
\quad x \in \domain,  
\quad y(x,\xi) = 0, \quad x \in \partial\domain,
\end{align*}
where $\partial\domain$ is the boundary of the domain $\domain$.
The set $\Xi$, the parameter $\alpha > 0$, the diffusion coefficient 
$\kappa : \domain \times \Xi \to (0,\infty)$
and the random right-hand side $r : \domain \times \Xi \to \real$
are defined in \Cref{example1,example2}.
Defining 
$\dualpHzeroone[\domain]{Bu}{v} = -\inner[L^2(\domain)]{u}{v}$,
and
\begin{align*}
\dualpHzeroone[\domain]{A(\xi)y}{v} &= 
\int_\domain \kappa(x,\xi)\nabla y(x) \cdot \nabla v(x) \du x, 
\\
\dualpHzeroone[\domain]{g(\xi)}{v} &= 
\int_\domain r(x,\xi)v(x) \du x,
\end{align*}
the weak form of the linear  PDE can be written
in the form provided in \eqref{eq:saa:leq}.

We approximate the control problem \eqref{eq:saa:lqcp}  
using a finite element discretization.
The control space $\csp =  L^2(\domain)$ is discretized
using piecewise constant
functions and the state space $\ssp =  H^1_0(\domain)$ 
is discretized using piecewise linear continuous 
functions defined on a triangular mesh
on	$[0,1]^2$
with $n \in \natural$ being the number of cells in each direction,
yielding finite element approximations of  
\eqref{eq:saa:lqcp} and corresponding SAA problems.
To simplify notation, we omit the index $n$ when referring to 
the solutions to these optimization problems. 
The dimension of the discretized control space is $2n^2$.

\begin{figure}[t]
	\centering
	\subfloat[\Cref{example1} with $n =  \nn$.\label{fig:nonsmooth}]{
		\includegraphics[width=0.5\textwidth]
		{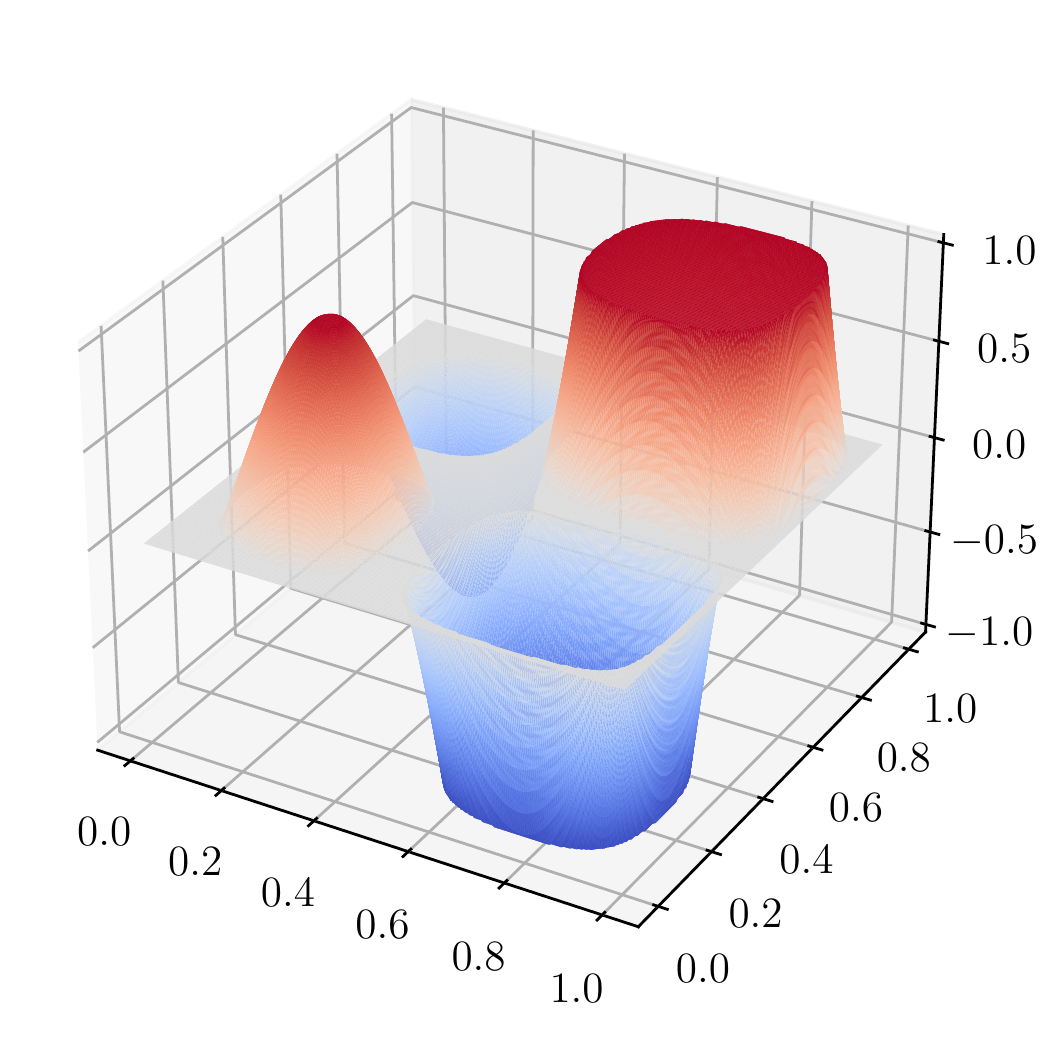}}
	\subfloat[\Cref{example2} with $n = \dn$.]{
		\includegraphics[width=0.48\textwidth]
		{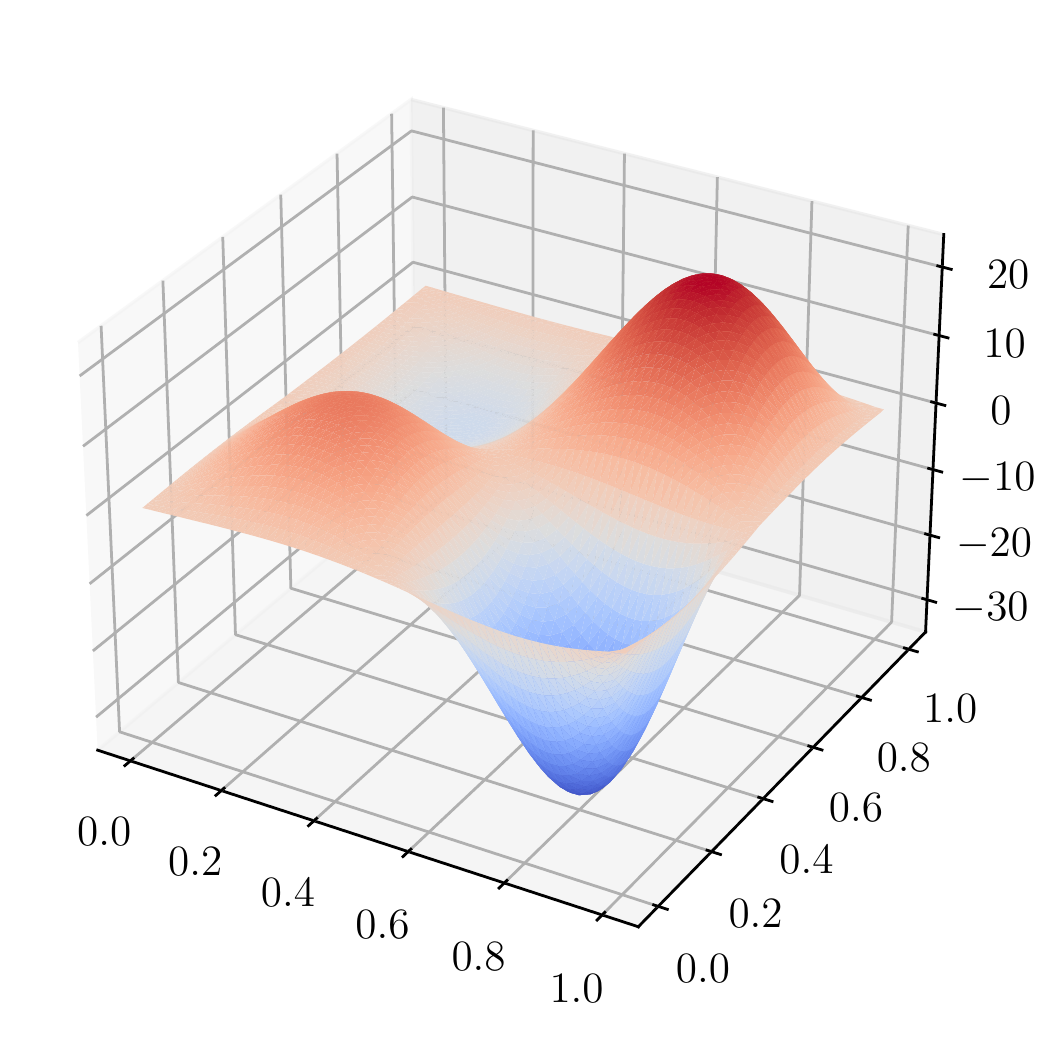}
	}
	\caption{Reference solutions.}
	\label{fig:reference_solutions}
\end{figure}

\begin{example}
	\label{example1}
	\normalfont 
	We define $\alpha = 10^{-3}$,
	$\Xi = [0.5, 3.5] \times [-1,1]$, 
	the random right-hand side
	$r(x,\xi) = \xi_2 \exp(2x_1)\sin(2\pi x_2)$, and
	$\kappa(\xi) = \xi_1$.
	The random variables $\xi_1$ and $\xi_2$ are
	independent, and $\xi_1$
	has a truncated normal distribution supported on $[0.5, 3.5]$
	with mean $2$ and standard deviation $0.25$
	(cf.\ \cite[p.\ 2092]{Geiersbach2019a}),
	and $\xi_2$ is uniformly distributed over $[-1,1]$.
	We choose 
	$\Psi(\cdot) = \gamma \norm[L^1(\domain)]{\cdot} + I_{\adcsp}(\cdot)$
	with $\gamma = 5.5 \cdot 10^{-4}$ and
	$\adcsp = \{\, u \in L^2(\domain) \colon \, -1 \leq u \leq 1 \, \}$,
	which is nonempty, closed, and convex \cite[p.\ 56]{Hinze2009}.
	Furthermore, let $n = \nn$.

	Since $\kappa(\xi) = \xi_1$ is a real-valued random variable, 
	we can evaluate  $\nabla \erpobj_1(u)$ and its empirical mean 
	using only two PDE solutions which can be shown by dividing
	\eqref{eq:saa:leq} by $\kappa(\xi)$.
	It allows us to compute the 
	solutions to the finite element approximation of \eqref{eq:saa:lqcp}
	and to their SAA problems with moderate computational effort even though
	$n = \nn$ is relatively large.

	We solved the finite element discretization 
	of \eqref{eq:saa:lqcp} and the SAA problems using a semismooth Newton
	method \cite{Stadler2009,Ulbrich2011,Pieper2015}
	applied to a normal map
	(cf.\ \cite[eq.\ (3.3)]{Pieper2015}),
	which provides a reformulation of the 
	first-order optimality conditions as a nonsmooth equation
	\cite[Sect.\ 3.1]{Pieper2015}. 
	The finite element discretization was performed using
	\texttt{FEniCs} \cite{Alnes2015,Logg2012}.
	Sparse linear systems were solved using a direct method.
\end{example}

\begin{example}
	\label{example2}
	\normalfont 
	We define $\alpha = 10^{-4}$, 
	$\Xi = [3,5] \times [0.5,2.5]$ and 
	the piecewise constant field
	$\kappa$ 
	by
	$\kappa(x,\xi) = \xi_1$
	if $x \in (0,1) \times (1/2,1)$
	and 
	$\kappa(x,\xi) = \xi_2$
	if $x \in (0,1) \times (1/2,1)$
	(cf.\ \cite[Example 3]{Geiersbach2019}).
	The random variables $\xi_1$ and $\xi_2$ are independent and uniformly
	distributed over $[3,5]$  and $[0.5,2.5]$, respectively.
	Moreover $r = 0$, $\Psi = 0$, and $n=\dn$.

	To obtain a deterministic reference solution to
	the finite element approximation of \eqref{eq:saa:lqcp}, 
	we approximate the probability distribution of $\xi$
	by a discrete uniform distribution.
	It is supported on the grid
	points of a uniform mesh of $\Xi$ using $50$ grid points
	in each direction, yielding a discrete distribution with $2500$ 
	scenarios. Samples for the SAA problems are generated from this
	discrete distribution.

	We used \href{http://www.dolfin-adjoint.org/}
	{\texttt{dolfin-adjoint}} 
	\cite{Farrell2013,Funke2013,Mitusch2019}
	with \texttt{FEniCs} \cite{Alnes2015,Logg2012} 
	to evaluate the SAA objective functions 
	and their derivatives, 	and solved the problems 	using 
	\href{https://github.com/funsim/moola}{\texttt{moola}}'s 
	\texttt{NewtonCG}\ method \cite{Funke2013,Schwedes2017}.
\end{example}

\begin{figure}[t]
	\centering
	\subfloat[\Cref{example1} with $n =  \nn$.]{
		\includegraphics[width=0.5\textwidth]
		{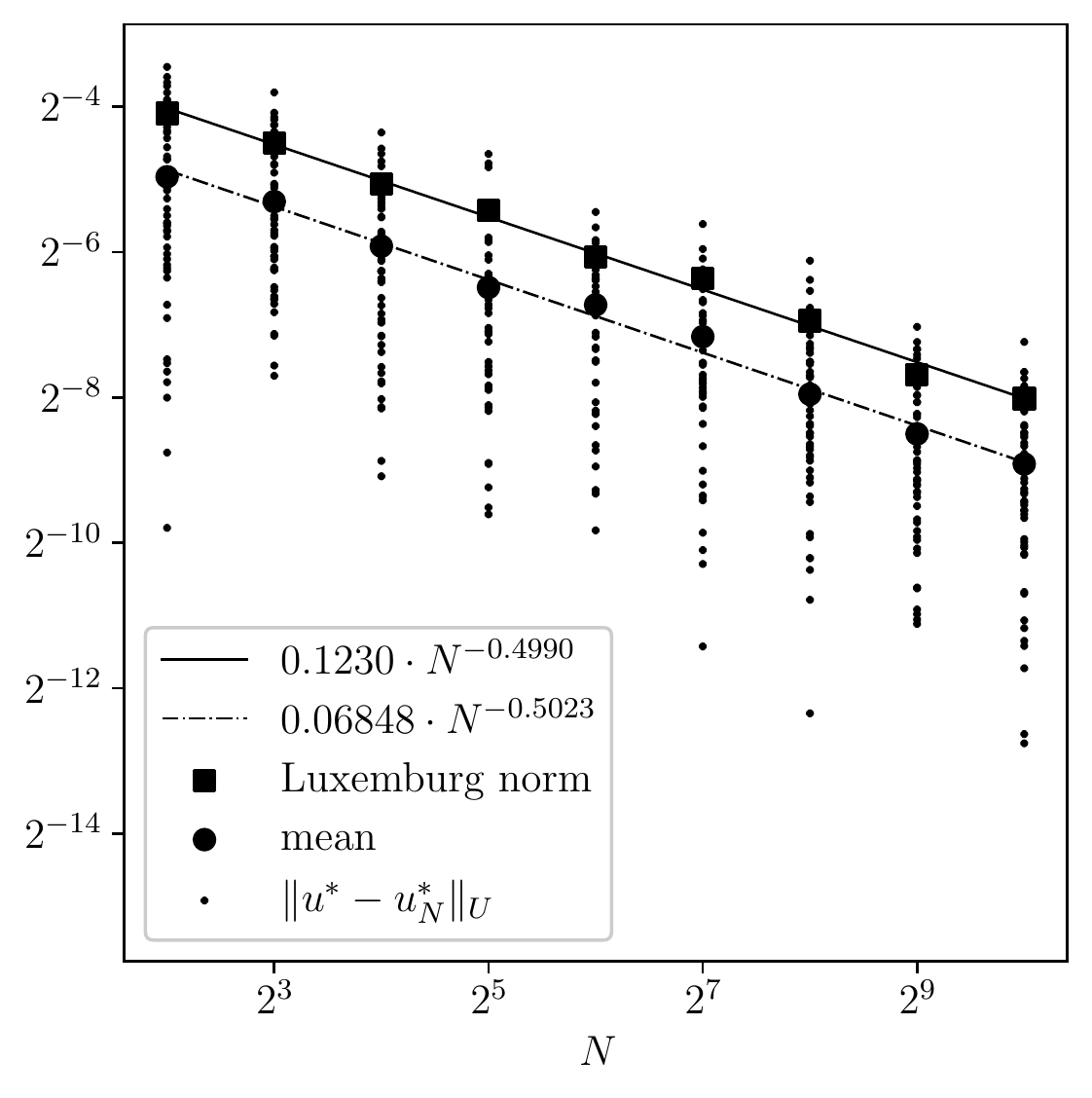}}
	\subfloat[\Cref{example2} with $n = \dn$.]{
		\includegraphics[width=0.4925\textwidth]
		{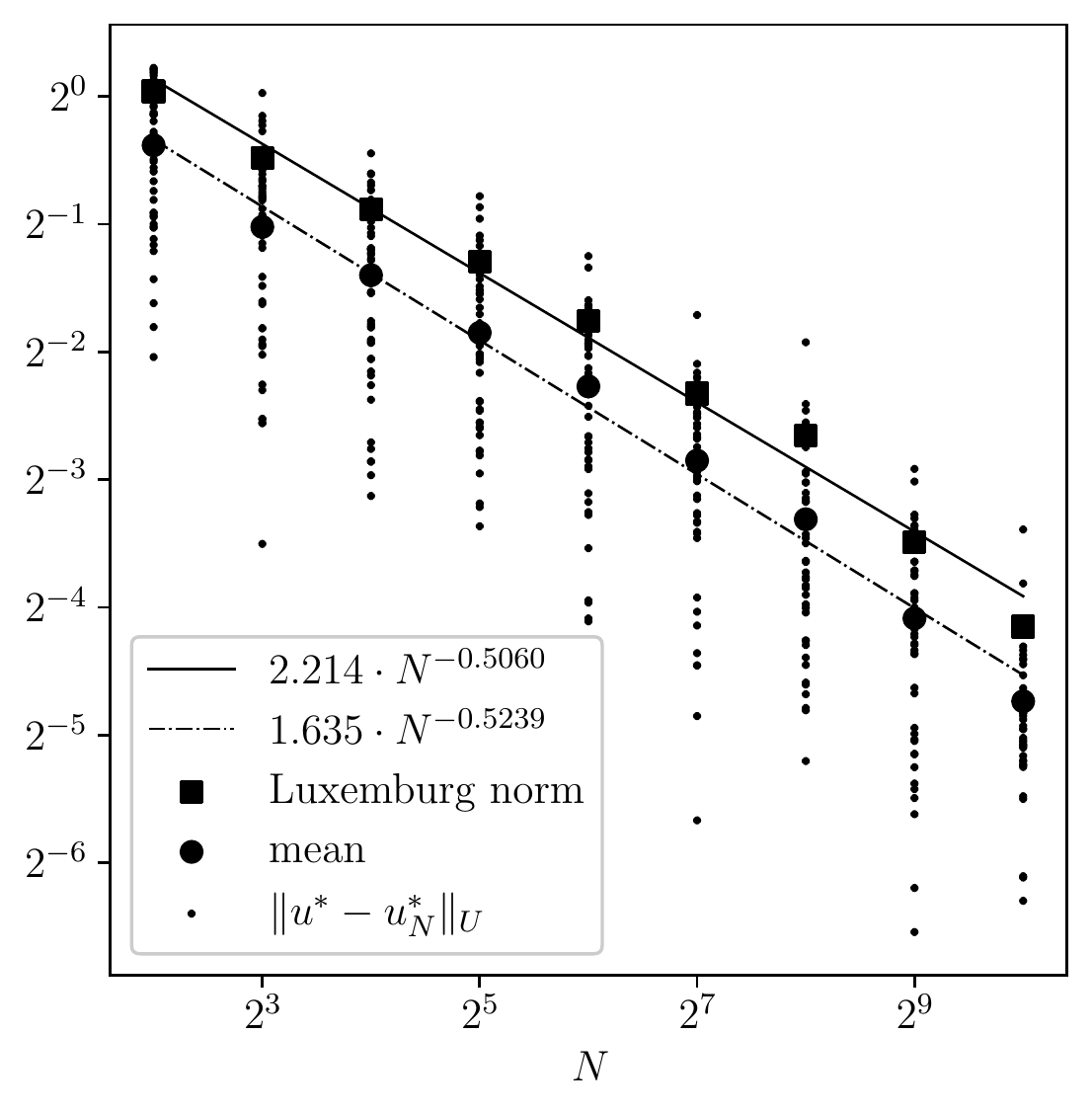}
	}
	\caption{For each example, $50$ independent realizations of 
		$\norm[\csp]{u^*-u_N^*}$, and
		the empirical mean  error and empirical Luxemburg norm.
		The convergence rates were computed using least squares.}
	\label{fig:convergence_rates}
\end{figure}

\Cref{fig:reference_solutions} depicts the
reference solutions for \Cref{example1,example2}.
To generate the surface plots depicted in \Cref{fig:reference_solutions}, 
the piecewise constant reference solutions were interpolated to the
space of piecewise linear continuous functions.

To illustrate  the convergence rate
$1/\sqrt{N}$ for $\cE{\norm[\csp]{u^*-u_N^*}}$, we generated $50$
independent samples of $\norm[\csp]{u^*-u_N^*}$
and computed the sample average. 
In order to empirically verify the exponential tail bound
\eqref{eq:saa:2020-01-25T19:33:40.055}, we use the
fact that it is equivalent
to a certain bound on the Luxemburg norm of 
$u^*-u_N^*$.
We define the Luxemburg norm $\luxemburg[\csp]{2}{\cdot}$
of a random vector $\rv : \Omega \to \csp$ by
\begin{align}
\label{eq:luxemburgnorm}
\luxemburg[\csp]{2}{Z}
= \inf_{\nu > 0} \, \{ \, \nu : \,
\cE{\phi(\norm[\csp]{Z}/\nu)} 
\leq 1 \, \},
\end{align}
where $\phi : \real \to \real$
is given by $\phi(x) = \exp(x^2)-1$,
and 
$L_{\phi}(\Omega; \csp) = L_{\phi(\norm[\csp]{\cdot})}(\Omega; \csp)$ is the 
Orlicz space consisting of each random vector $\rv : \Omega \to \csp$
such that there exists $\nu > 0$
with $\cE{\phi(\norm[\csp]{Z}/\nu)} < \infty$;
cf.\ \cite[Sect.\ 6.2]{Kosmol2011}.
The exponential tail bound \eqref{eq:saa:2020-01-25T19:33:40.055} implies 
\begin{align}
\label{eq:lux}
\luxemburg[\csp]{2}{u^*-u_N^*} 
\leq \frac{3\sqrt{3}\tau}{\alpha \sqrt{N}},
\end{align}
and
\eqref{eq:lux} 
ensures
$\Prob{\norm[\csp]{u^*-u_N^*} \geq \varepsilon}
\leq 2 \eu^{-\tau^{-2} N \varepsilon^2 \alpha^2/27}$
for all $\varepsilon > 0$. 
These two statements follow from
\cite[Thm.\ 3.4 on p.\ 56]{Buldygin2000} when applied to 
the real-valued random variable $\norm[\csp]{u^*-u_N^*}$. 
To empirically verify the convergence rate
$1/\sqrt{N}$ for $\luxemburg[\csp]{2}{u^*-u_N^*}$,
we approximated the expectation in \eqref{eq:luxemburgnorm}
using the same samples used to estimate $\cE{\norm[\csp]{u^*-u_N^*}}$.

\Cref{fig:convergence_rates} depicts
$50$ realizations of the errors $\norm[\csp]{u^*-u_N^*}$, 
the empirical approximations of
$\cE{\norm[\csp]{u^*-u_N^*}}$
and of the Luxemburg norm $\luxemburg[\csp]{2}{u^*-u_N^*}$
as well as the corresponding
convergence rates.
The rates were computed using least squares. 
The empirical convergences rates depicted in \Cref{fig:convergence_rates}
are close to  the theoretical rate
$1/\sqrt{N}$ for $\cE{\norm[\csp]{u^*-u_N^*}}$
and $\luxemburg[\csp]{2}{u^*-u_N^*}$; 
see \eqref{eq:Feb2820211612} and \eqref{eq:lux}.

\section{Discussion}
\label{sec:discussion}

We have considered convex stochastic programs posed in Hilbert spaces
where the integrand is strongly convex with the same
parameter for each random element's realization.
We have established exponential tail bounds for the distance
between SAA solutions and the true ones.
For this problem class, tail bounds are optimal up to
problem-independent, moderate constants. 
We have applied our findings to stochastic linear-quadratic control problems, 
a subclass of the above problem class.

We conclude the paper by 
illustrating that the ``dynamics'' of finite-
and infinite-dimensional stochastic programs can be quite different.
We consider
\begin{align}
\label{eq:exone}
\min_{\norm{x} \leq 1}\, \cE{\norm{x}^2-2x^T\zeta},
\end{align}
where $\zeta$ is an $\real^n$-valued, mean-zero Gaussian
random vector with covariance matrix $\sigma^2I$ and $\sigma^2 > 0$. 
This corresponds to the choice $m = 1$ in \cite[Ex.\ 1]{Shapiro2008}.
For $\delta \in (0, 0.3)$ and $\varepsilon \in (0, 1)$,
at least $N > n\sigma^2/\varepsilon = \cE{\norm{\zeta}^2}/\varepsilon$
samples are required for the 
corresponding SAA problem's optimal solution
to be an $\varepsilon$-optimal solution to 
\eqref{eq:exone} 
with a probability of at least $1-\delta$
\cite[Ex.\ 1]{Shapiro2008}.

The infinite-dimensional analogue of \eqref{eq:exone} is given by
\begin{align}
\label{eq:exoneinfinite}
\min_{\norm[\ell^2(\natural)]{u} \leq 1}\, 
\cE{\norm[\ell^2(\natural)]{u}^2-2\innerp{\ell^2(\natural)}{u}{\xi}},
\end{align}
where $\xi$ is an $\ell^2(\natural)$-valued,
mean-zero Gaussian random vector, and
$\ell^2(\natural)$ is the standard sequence space.
For each $\varepsilon \in (0, 1)$, 
the SAA solution  $u_N^*$ corresponding to \eqref{eq:exoneinfinite} is
an $\varepsilon$-optimal solution 
to \eqref{eq:exoneinfinite} 
if and only if we have
$\norm[\ell^2(\natural)]{(1/N)\sum_{i=1}^N \xi^i}^2 \leq \varepsilon$.
Combining \Cref{prop:saa:2020-11-21T20:25:01.71}
with \cite[Rem.\ 4]{Pinelis1986}, we find that
$N \geq (3/\varepsilon)\ln(2/\delta)\cE{\norm[\ell^2(\natural)]{\xi}^2}$
samples are sufficient in order for 
$u_N^*$ to be an $\varepsilon$-optimal solution
to \eqref{eq:exoneinfinite}, with  a probability of at least 
$1-\delta \in (0,1)$.

Let us compare the stochastic program
\eqref{eq:exone} with \eqref{eq:exoneinfinite}.
Whereas $\cE{\norm{\zeta}^2} = n \sigma^2 \to \infty$
as $n \to \infty$ and $\cE{|\zeta_k|^2}=\sigma^2$
($1\leq k\leq n$), we have
$\cE{\norm[\ell^2(\natural)]{\xi}^2}< \infty$
due to the Landau--Shepp--Fernique theorem 
and $\cE{|\xi_k|^2} \to 0$ as $k \to \infty$
\cite[p.\ 59]{Yurinsky1995}.
We find that the ``overall level-of-randomness'' for the finite-dimensional
problem \eqref{eq:exone} depends on its dimension $n$, while
that for the infinite-dimensional analogue  \eqref{eq:exoneinfinite} 
is fixed.

\subsection*{Acknowledgments}
	\small{%
		Parts of the presented results originate from the author's
	dissertation (submitted January, 2021). 
	The author thanks his committee members, 
	Professor Michael Ulbrich, Professor Karl Kunisch, 
	and Professor Alexander Shapiro, for their valuable comments
	on the dissertation. The author
	thanks Niklas Behringer for his constructive comments 
	on an earlier draft of the manuscript.}

\bibliography{convexsaa} 

\begin{thebibliography}{10}

\bibitem{Agarwal2009}
{\sc A.~Agarwal, P.~L. Bartlett, P.~Ravikumar, and M.~J. Wainwright}, {\em
  Information-theoretic lower bounds on the oracle complexity of stochastic
  convex optimization}, IEEE Trans. Inform. Theory, 58 (2012), pp.~3235--3249,
  \url{https://doi.org/10.1109/TIT.2011.2182178}.

\bibitem{Ali2017}
{\sc A.~A. Ali, E.~Ullmann, and M.~Hinze}, {\em Multilevel {M}onte {C}arlo
  analysis for optimal control of elliptic {PDEs} with random coefficients},
  SIAM/ASA J. Uncertainty Quantification, 5 (2017), pp.~466--492,
  \url{https://doi.org/10.1137/16M109870X}.

\bibitem{Alnes2015}
{\sc M.~S. Aln{\ae}s, J.~Blechta, J.~Hake, A.~Johansson, B.~Kehlet, A.~Logg,
  C.~Richardson, J.~Ring, M.~E. Rognes, and G.~N. Wells}, {\em The {FE}ni{CS}
  project version 1.5}, Arch. Numer. Software, 3 (2015), pp.~9--23,
  \url{https://doi.org/10.11588/ans.2015.100.20553}.

\bibitem{Artstein1995}
{\sc Z.~Artstein and R.~J.~B. Wets}, {\em Consistency of minimizers and the
  {SLLN} for stochastic programs}, J. Convex Anal., 2 (1995), pp.~1--17.

\bibitem{Aubin2009}
{\sc J.-P. Aubin and H.~Frankowska}, {\em Set-{V}alued {A}nalysis}, Mod.
  Birkh\"auser Class., Springer, Boston, MA, 2009,
  \url{https://doi.org/10.1007/978-0-8176-4848-0}.

\bibitem{Bach2011}
{\sc F.~Bach and E.~Moulines}, {\em Non-asymptotic analysis of stochastic
  approximation algorithms for machine learning}, in Advances in Neural
  Information Processing Systems, J.~Shawe-Taylor, R.~Zemel, P.~Bartlett,
  F.~Pereira, and K.~Q. Weinberger, eds., Red Hook, NY, 2011, Curran
  Associates, Inc., pp.~451--459,
  \url{https://papers.nips.cc/paper/4316-non-asymptotic-analysis-of-stochastic-approximation-algorithms-for-machine-learning}.

\bibitem{Banholzer2019}
{\sc D.~Banholzer, J.~Fliege, and R.~Werner}, {\em On rates of convergence for
  sample average approximations in the almost sure sense and in mean}, Math.
  Program.,  (2019), \url{https://doi.org/10.1007/s10107-019-01400-4}.

\bibitem{Bergounioux2000}
{\sc M.~Bergounioux, M.~Haddou, M.~Hinterm\"uller, and K.~Kunisch}, {\em A
  comparison of a {M}oreau--{Y}osida-based active set strategy and interior
  point methods for constrained optimal control problems}, SIAM J. Optim., 11
  (2000), pp.~495--521, \url{https://doi.org/10.1137/S1052623498343131}.

\bibitem{Bharucha-Reid1972}
{\sc A.~T. Bharucha-Reid}, {\em Random {I}ntegral {E}quations}, Math. Sci. Eng.
  96, Academic Press, New York, 1972.

\bibitem{Bonnans2013}
{\sc J.~F. Bonnans and A.~Shapiro}, {\em Perturbation {A}nalysis of
  {O}ptimization {P}roblems}, Springer Ser. Oper. Res., Springer, New York, NY,
  2000, \url{https://doi.org/10.1007/978-1-4612-1394-9}.

\bibitem{Buldygin2000}
{\sc V.~V. Buldygin and {\relax Yu}.~V. Kozachenko}, {\em Metric
  {C}haracterization of {R}andom {V}ariables and {R}andom {P}rocesses}, Transl.
  Math. Monogr. 188, American Mathematical Society, Providence, RI, 2000.

\bibitem{Castaing1977}
{\sc C.~Castaing and M.~Valadier}, {\em Convex {A}nalysis and {M}easurable
  {M}ultifunctions}, Lecture Notes in Math. 580, Springer, Berlin, 1977,
  \url{https://doi.org/10.1007/bfb0087685}.

\bibitem{Charrier2012}
{\sc J.~Charrier}, {\em Strong and weak error estimates for elliptic partial
  differential equations with random coefficients}, SIAM J. Numer. Anal., 50
  (2012), pp.~216--246, \url{https://doi.org/10.1137/100800531}.

\bibitem{Conway1990}
{\sc J.~B. Conway}, {\em A {C}ourse in {F}unctional {A}nalysis}, Grad. Texts in
  Math. 96, Springer, New York, 2nd~ed., 1990,
  \url{https://doi.org/10.1007/978-1-4757-4383-8}.

\bibitem{Duchi2012}
{\sc J.~C. Duchi, P.~L. Bartlett, and M.~J. Wainwright}, {\em Randomized
  smoothing for stochastic optimization}, SIAM J. Optim., 22 (2012),
  pp.~674--701, \url{https://doi.org/10.1137/110831659}.

\bibitem{Farrell2013}
{\sc P.~Farrell, D.~Ham, S.~Funke, and M.~Rognes}, {\em Automated derivation of
  the adjoint of high-level transient finite element programs}, SIAM J. Sci.
  Comput., 35 (2013), pp.~C369--C393, \url{https://doi.org/10.1137/120873558}.

\bibitem{Fukuda1990}
{\sc R.~Fukuda}, {\em Exponential integrability of sub-{G}aussian vectors},
  Probab. Theory Related Fields, 85 (1990), pp.~505--521,
  \url{https://doi.org/10.1007/BF01203168}.

\bibitem{Funke2013}
{\sc S.~W. Funke and P.~E. Farrell}, {\em A framework for automated
  {PDE}-constrained optimisation}, preprint,
  \url{https://arxiv.org/abs/1302.3894}, 2013.

\bibitem{Garreis2019b}
{\sc S.~Garreis, T.~M. Surowiec, and M.~Ulbrich}, {\em An interior-point
  approach for solving risk-averse {PDE}-constrained optimization problems with
  coherent risk measures}, SIAM J. Optim., 31 (2021), pp.~1--19,
  \url{https://doi.org/10.1137/19M125039X}.

\bibitem{Garreis2017}
{\sc S.~Garreis and M.~Ulbrich}, {\em Constrained optimization with low-rank
  tensors and applications to parametric problems with {PDE}s}, SIAM J. Sci.
  Comput., 39 (2017), pp.~A25--A54, \url{https://doi.org/10.1137/16M1057607}.

\bibitem{Ge2018}
{\sc L.~Ge, L.~Wang, and Y.~Chang}, {\em A sparse grid stochastic collocation
  upwind finite volume element method for the constrained optimal control
  problem governed by random convection diffusion equations}, J. Sci. Comput.,
  77 (2018), pp.~524--551, \url{https://doi.org/10.1007/s10915-018-0713-y}.

\bibitem{Geiersbach2019a}
{\sc C.~Geiersbach and G.~{\relax Ch}. Pflug}, {\em Projected stochastic
  gradients for convex constrained problems in {H}ilbert spaces}, SIAM J.
  Optim., 29 (2019), pp.~2079--2099, \url{https://doi.org/10.1137/18M1200208}.

\bibitem{Geiersbach2020}
{\sc C.~Geiersbach and T.~Scarinci}, {\em Stochastic proximal gradient methods
  for nonconvex problems in {H}ilbert spaces}, Comput. Optim. Appl., 78 (2021),
  pp.~705--740, \url{https://doi.org/10.1007/s10589-020-00259-y}.

\bibitem{Geiersbach2019}
{\sc C.~Geiersbach and W.~Wollner}, {\em A stochastic gradient method with mesh
  refinement for {PDE}-constrained optimization under uncertainty}, SIAM J.
  Sci. Comput., 42 (2020), pp.~A2750--A2772,
  \url{https://doi.org/10.1137/19M1263297}.

\bibitem{Guigues2017}
{\sc V.~Guigues, A.~Juditsky, and A.~Nemirovski}, {\em Non-asymptotic
  confidence bounds for the optimal value of a stochastic program}, Optim.
  Methods Softw., 32 (2017), pp.~1033--1058,
  \url{https://doi.org/10.1080/10556788.2017.1350177}.

\bibitem{Guth2019}
{\sc P.~A. Guth, V.~Kaarnioja, F.~Y. Kuo, C.~Schillings, and I.~H. Sloan}, {\em
  A quasi-{M}onte {C}arlo method for optimal control under uncertainty},
  SIAM/ASA J. Uncertain. Quantif., 9 (2021), pp.~354--383,
  \url{https://doi.org/10.1137/19M1294952}.

\bibitem{Hille1957}
{\sc E.~Hille and R.~S. Phillips}, {\em Functional {A}nalysis and
  {S}emi-{G}roups}, Colloq. Publ. 31, American Mathematical Society,
  Providence, RI, 1957.

\bibitem{Hinze2005a}
{\sc M.~Hinze}, {\em A variational discretization concept in control
  constrained optimization: {T}he linear-quadratic case}, Comput. Optim. Appl.,
  30 (2005), pp.~45--61, \url{https://doi.org/10.1007/s10589-005-4559-5}.

\bibitem{Hinze2009}
{\sc M.~Hinze, R.~Pinnau, M.~Ulbrich, and S.~Ulbrich}, {\em Optimization with
  {PDE} {C}onstraints}, Math. Model. Theory Appl. 23, Springer, Dordrecht,
  2009, \url{https://doi.org/10.1007/978-1-4020-8839-1}.

\bibitem{Hoffhues2020}
{\sc M.~Hoffhues, W.~R{\"o}misch, and T.~M. Surowiec}, {\em On quantitative
  stability in infinite-dimensional optimization under uncertainty}, Optim.
  Lett., 15 (2021), pp.~2733--2756,
  \url{https://doi.org/10.1007/s11590-021-01707-2}.

\bibitem{Hytoenen2016}
{\sc T.~Hyt\"{o}nen, J.~van Neerven, M.~Veraar, and L.~Weis}, {\em Analysis in
  {B}anach {S}paces: {M}artingales and {L}ittlewood-{P}aley {T}heory}, Ergeb.
  Math. Grenzgeb. (3) 63, Springer, Cham, 2016,
  \url{https://doi.org/10.1007/978-3-319-48520-1}.

\bibitem{Ito2008}
{\sc K.~Ito and K.~Kunisch}, {\em Lagrange {M}ultiplier {A}pproach to
  {V}ariational {P}roblems and {A}pplications}, Adv. Des. Control 15, SIAM,
  Philadelphia, PA, 2008, \url{https://doi.org/10.1137/1.9780898718614}.

\bibitem{Kosmol2011}
{\sc P.~Kosmol and D.~M\"{u}ller-Wichards}, {\em Optimization in {F}unction
  {S}paces: {W}ith {S}tability {C}onsiderations in {O}rlicz Spaces}, De Gruyter
  Ser. Nonlinear Anal. Appl. 13, De Gruyter, Berlin, 2011,
  \url{https://doi.org/10.1515/9783110250213}.

\bibitem{Kouri2013}
{\sc D.~P. Kouri, M.~Heinkenschloss, D.~Ridzal, and B.~van Bloemen~Waanders},
  {\em A trust-region algorithm with adaptive stochastic collocation for {PDE}
  optimization under uncertainty}, SIAM J. Sci. Comput., 35 (2013),
  pp.~A1847--A1879, \url{https://doi.org/10.1137/120892362}.

\bibitem{Kouri2018}
{\sc D.~P. Kouri and A.~Shapiro}, {\em Optimization of {PDE}s with {U}ncertain
  {I}nputs}, in Frontiers in PDE-Constrained Optimization, H.~Antil, D.~P.
  Kouri, M.-D. Lacasse, and D.~Ridzal, eds., IMA Vol. Math. Appl. 163,
  Springer, New York, NY, 2018, pp.~41--81,
  \url{https://doi.org/10.1007/978-1-4939-8636-1_2}.

\bibitem{Kouri2018a}
{\sc D.~P. Kouri and T.~M. Surowiec}, {\em Existence and optimality conditions
  for risk-averse {PDE}-constrained optimization}, SIAM/ASA J. Uncertainty
  Quantification, 6 (2018), pp.~787--815,
  \url{https://doi.org/10.1137/16M1086613}.

\bibitem{Kreyszig1978}
{\sc E.~Kreyszig}, {\em Introductory {F}unctional {A}nalysis with
  {A}pplications}, John Wiley \& Sons, New York, NY, 1978.

\bibitem{Lan2020}
{\sc G.~Lan}, {\em First-order and {S}tochastic {O}ptimization {M}ethods for
  {M}achine {L}earning}, Springer Ser. Data Sci., Springer, Cham, 2020,
  \url{https://doi.org/10.1007/978-3-030-39568-1}.

\bibitem{Logg2012}
{\sc A.~Logg, K.-A. Mardal, and G.~N. Wells}, eds., {\em Automated {S}olution
  of {D}ifferential {E}quations by the {F}inite {E}lement {M}ethod: {T}he
  {FEniCS} {B}ook}, Lect. Notes Comput. Sci. Eng. 84, Springer, Berlin, 2012,
  \url{https://doi.org/10.1007/978-3-642-23099-8}.

\bibitem{Marin2018}
{\sc F.~J. Mar{\'i}n, J.~Mart{\'i}nez-Frutos, and F.~Periago}, {\em Control of
  random {PDE}s: {A}n overview}, in Recent Advances in PDEs: Analysis, Numerics
  and Control: In Honor of Prof. Fern{\'a}ndez-Cara's 60th Birthday,
  A.~Doubova, M.~Gonz{\'a}lez-Burgos, F.~Guill{\'e}n-Gonz{\'a}lez, and
  M.~Mar{\'i}n~Beltr{\'a}n, eds., Springer, Cham, 2018, pp.~193--210,
  \url{https://doi.org/10.1007/978-3-319-97613-6_10}.

\bibitem{Martin2021}
{\sc M.~Martin, S.~Krumscheid, and F.~Nobile}, {\em Complexity analysis of
  stochastic gradient methods for {PDE}-constrained optimal control problems
  with uncertain parameters}, ESAIM Math. Model. Numer. Anal., 55 (2021),
  pp.~1599--1633, \url{https://doi.org/10.1051/m2an/2021025}.

\bibitem{MartnezFrutos2018}
{\sc J.~Mart{\'{\i}}nez-Frutos and F.~P. Esparza}, {\em Optimal {C}ontrol of
  {PDEs} under {U}ncertainty: {A}n {I}ntroduction with {A}pplication to
  {O}ptimal {S}hape {D}esign of {S}tructures}, SpringerBriefs Math., Springer,
  Cham, 2018, \url{https://doi.org/10.1007/978-3-319-98210-6}.

\bibitem{Mitusch2019}
{\sc S.~K. Mitusch, S.~W. Funke, and J.~S. Dokken}, {\em dolfin-adjoint 2018.1:
  automated adjoints for {FEniCS} and {F}iredrake}, J. Open Source Softw., 4
  (2019), p.~1292, \url{https://doi.org/10.21105/joss.01292}.

\bibitem{Nemirovski2009}
{\sc A.~Nemirovski, A.~Juditsky, G.~Lan, and A.~Shapiro}, {\em Robust
  {S}tochastic {A}pproximation {A}pproach to {S}tochastic {P}rogramming}, SIAM
  J. Optim., 19 (2009), pp.~1574--1609,
  \url{https://doi.org/10.1137/070704277}.

\bibitem{Nemirovsky1983}
{\sc A.~S. Nemirovsky and D.~B. Yudin}, {\em Problem {C}omplexity and {M}ethod
  {E}fficiency in {O}ptimization}, Wiley-Interscience Series in Discrete
  Mathematics, John Wiley \& Sons, Chichester, 1983.
\newblock Translated by E. R. Dawson.

\bibitem{Pieper2015}
{\sc K.~Pieper}, {\em Finite element discretization and efficient numerical
  solution of elliptic and parabolic sparse control problems}, {D}issertation,
  Technische Universit\"at M\"unchen, M\"unchen, 2015,
  \url{http://nbn-resolving.de/urn/resolver.pl?urn:nbn:de:bvb:91-diss-20150420-1241413-1-4}.

\bibitem{Pinelis1994}
{\sc I.~Pinelis}, {\em Optimum bounds for the distributions of martingales in
  {B}anach spaces}, Ann. Probab., 22 (1994), pp.~1679--1706,
  \url{https://doi.org/10.1214/aop/1176988477}.

\bibitem{Pinelis1986}
{\sc I.~F. Pinelis and A.~I. Sakhanenko}, {\em Remarks on {I}nequalities for
  {L}arge {D}eviation {P}robabilities}, Theory Probab. Appl., 30 (1986),
  pp.~143--148, \url{https://doi.org/10.1137/1130013}.

\bibitem{Roemisch2021}
{\sc W.~R\"omisch and T.~M. Surowiec}, {\em Asymptotic properties of {M}onte
  {C}arlo methods in elliptic {PDE}-constrained optimization under
  uncertainty}, preprint, \url{http://arxiv.org/abs/2106.06347}, 2021.

\bibitem{Royset2019}
{\sc J.~O. Royset}, {\em Approximations of semicontinuous functions with
  applications to stochastic optimization and statistical estimation}, Math.
  Program., 184 (2020), pp.~289--318,
  \url{https://doi.org/10.1007/s10107-019-01413-z}.

\bibitem{Schwedes2017}
{\sc T.~Schwedes, D.~A. Ham, S.~W. Funke, and M.~D. Piggott}, {\em {M}esh
  {D}ependence in {PDE}-{C}onstrained {O}ptimisation}, SpringerBriefs Math.
  Planet Earth, Springer, Cham, 2017,
  \url{https://doi.org/10.1007/978-3-319-59483-5}.

\bibitem{Shalev-Shwartz2010}
{\sc S.~Shalev-Shwartz, O.~Shamir, N.~Srebro, and K.~Sridharan}, {\em
  Learnability, stability and uniform convergence}, J. Mach. Learn. Res., 11
  (2010), pp.~2635--2670.

\bibitem{Shapiro1993}
{\sc A.~Shapiro}, {\em Asymptotic behavior of optimal solutions in stochastic
  programming}, Math. Oper. Res., 18 (1993), pp.~829--845,
  \url{https://doi.org/10.1287/moor.18.4.829}.

\bibitem{Shapiro2003}
{\sc A.~Shapiro}, {\em Monte {C}arlo {S}ampling {M}ethods}, in Stochastic
  {P}rogramming, Handbooks in Oper. Res. Manag. Sci. 10, Elsevier, 2003,
  pp.~353--425, \url{https://doi.org/10.1016/S0927-0507(03)10006-0}.

\bibitem{Shapiro2008}
{\sc A.~Shapiro}, {\em Stochastic programming approach to optimization under
  uncertainty}, Math. Program., 112 (2008), pp.~183--220,
  \url{https://doi.org/10.1007/s10107-006-0090-4}.

\bibitem{Shapiro2014}
{\sc A.~Shapiro, D.~Dentcheva, and A.~Ruszczy{\'{n}}ski}, {\em Lectures on
  {S}tochastic {P}rogramming: {M}odeling and {T}heory}, MOS-SIAM Ser. Optim.,
  SIAM, Philadelphia, PA, 2nd~ed., 2014,
  \url{https://doi.org/10.1137/1.9781611973433}.

\bibitem{Shapiro2005}
{\sc A.~Shapiro and A.~Nemirovski}, {\em On complexity of stochastic
  programming problems}, in Continuous Optimization: Current Trends and Modern
  Applications, V.~Jeyakumar and A.~Rubinov, eds., Appl. Optim. 99, Springer,
  Boston, MA, 2005, pp.~111--146,
  \url{https://doi.org/10.1007/0-387-26771-9_4}.

\bibitem{Stadler2009}
{\sc G.~Stadler}, {\em Elliptic optimal control problems with {$L^1$}-control
  cost and applications for the placement of control devices}, Comput. Optim.
  Appl., 44 (2009), pp.~159--181,
  \url{https://doi.org/10.1007/s10589-007-9150-9}.

\bibitem{Sun2015}
{\sc T.~Sun, W.~Shen, B.~Gong, and W.~Liu}, {\em A priori error estimate of
  stochastic {G}alerkin method for optimal control problem governed by
  stochastic elliptic {PDE} with constrained control}, J. Sci. Comput., 67
  (2015), pp.~405--431, \url{https://doi.org/10.1007/s10915-015-0091-7}.

\bibitem{Troeltzsch2010}
{\sc F.~Tr{\"o}ltzsch}, {\em On finite element error estimates for optimal
  control problems with elliptic {PDE}s}, in Large-Scale Scientific Computing,
  I.~Lirkov, S.~Margenov, and J.~Wa{\'{s}}niewski, eds., Lecture Notes in
  Comput. Sci. 5910, Berlin, 2010, Springer, pp.~40--53,
  \url{https://doi.org/10.1007/978-3-642-12535-5_4}.

\bibitem{Tsybakov1981}
{\sc A.~B. Tsybakov}, {\em Error bounds for the method of minimization of
  empirical risk}, Probl. Peredachi Inf., 17 (1981), pp.~50--61,
  \url{http://mi.mathnet.ru/ppi1380}.
\newblock (in Russian).

\bibitem{Ulbrich2011}
{\sc M.~Ulbrich}, {\em Semismooth {N}ewton {M}ethods for {V}ariational
  {I}nequalities and {C}onstrained {O}ptimization {P}roblems in {F}unction
  {S}paces}, MOS-SIAM Ser. Optim., SIAM, Philadelphia, PA, 2011,
  \url{https://doi.org/10.1137/1.9781611970692}.

\bibitem{Vakhania1987}
{\sc N.~N. Vakhania, V.~I. Tarieladze, and S.~A. Chobanyan}, {\em Probability
  {D}istributions on {B}anach spaces}, Math. Appl. (Soviet Ser.) 14, D. Reidel
  Publishing Co., Dordrecht, 1987,
  \url{https://doi.org/10.1007/978-94-009-3873-1}.
\newblock Translated from the Russian and with a preface by Wojbor A.
  Woyczynski.

\bibitem{Yurinsky1995}
{\sc V.~Yurinsky}, {\em Sums and {G}aussian {V}ectors}, Lecture Notes in Math.
  1617, Springer, Berlin, 1995, \url{https://doi.org/10.1007/BFb0092599}.

\end{thebibliography}
\end{document}